\documentclass[a4paper]{amsart}
\usepackage{amssymb, enumitem}
\usepackage[all]{xy}
\usepackage{hyperref, aliascnt}
\usepackage{mathtools}
\usepackage{xcolor}
\usepackage[english]{babel}
\usepackage{bbm}
\theoremstyle{definition}
\newtheorem{lma}{Lemma}[section]

\newaliascnt{thmCt}{lma}
\newtheorem{thm}[thmCt]{Theorem}
\aliascntresetthe{thmCt}

\newaliascnt{corCt}{lma}

\aliascntresetthe{corCt}

\newaliascnt{cnjCt}{lma}

\aliascntresetthe{cnjCt}

\newaliascnt{propCt}{lma}
\newtheorem{prop}[propCt]{Proposition}
\aliascntresetthe{propCt}

\newtheorem*{thm*}{Theorem}
\newtheorem*{cor*}{Corollary}
\newtheorem*{prop*}{Proposition}

\newcounter{theoremintro}

\newtheorem{thmintro}[theoremintro]{Theorem}

\newaliascnt{pgrCt}{lma}

\aliascntresetthe{pgrCt}

\newaliascnt{dfCt}{lma}
\newtheorem{df}[dfCt]{Definition}
\aliascntresetthe{dfCt}

\newaliascnt{remCt}{lma}
\newtheorem{rem}[remCt]{Remark}
\aliascntresetthe{remCt}

\newaliascnt{remsCt}{lma}

\aliascntresetthe{remsCt}

\newaliascnt{egCt}{lma}
\newtheorem{eg}[egCt]{Example}
\aliascntresetthe{egCt}

\newaliascnt{egsCt}{lma}

\aliascntresetthe{egsCt}

\newaliascnt{qstCt}{lma}

\aliascntresetthe{qstCt}

\newaliascnt{pbmCt}{lma}

\aliascntresetthe{pbmCt}

\newaliascnt{notaCt}{lma}
\newtheorem{nota}[notaCt]{Notation}
\aliascntresetthe{notaCt}

\newcommand{\beq}{\begin{equation}}
\newcommand{\eeq}{\end{equation}}
\newcommand{\beqa}{\begin{eqnarray*}}
\newcommand{\eeqa}{\end{eqnarray*}}
\newcommand{\bal}{\begin{align*}}
\newcommand{\eal}{\end{align*}}
\newcommand{\bi}{\begin{itemize}}
\newcommand{\ei}{\end{itemize}}
\newcommand{\be}{\begin{enumerate}}
\newcommand{\ee}{\end{enumerate}}

\newcommand{\ep}{\varepsilon}

\newcommand{\Z}{{\mathbb{Z}}}
\newcommand{\C}{{\mathbb{C}}}
\newcommand{\N}{{\mathbb{N}}}

\newcommand{\sr}{{\mathrm{sr}}}
\newcommand{\dr}{{\mathrm{dr}}}

\newcommand{\id}{{\mathrm{id}}}

\newcommand{\Aut}{{\mathrm{Aut}}}

\newcommand{\dimRok}{{\mathrm{dim}_\mathrm{Rok}}}
\newcommand{\cdimRok}{{\mathrm{dim}_\mathrm{Rok}^\mathrm{c}}}
\newcommand{\dimnuc}{{\mathrm{dim}_\mathrm{nuc}}}

\newcommand{\ca}{$C^*$-algebra}

\newcommand{\uca}{unital $C^*$-algebra}

\newcommand{\I}{\infty}

\title[]{Partial C*-dynamics and Rokhlin dimension}

\thanks{The second named author was partially supported
by the Deutsche Forschungsgemeinschaft's (DFG, German Research Foundation) eigene Stelle, and by a Postdoctoral Research Fellowship
from the Humboldt Foundation. The third named author was supported by a Kreitman Foundation Fellowship, a Minerva
Fellowship Programme, and by an Israel Science Foundation grant no.~476/16. The second and third named authors were partially supported by the 
DFG through SFB 878 and under Germany's Excellence Strategy  EXC 2044 -- 390685587, Mathematics M\"unster -- Dynamics -- Geometry -- Structure, and also by ERC Advanced Grant 
834267 - AMAREC}

\author[Fernando Abadie]{Fernando Abadie}
\address{Fernando Abadie
Centro de Matem\'atica, Facultad de Ciencias,
Universidad de la Rep\'ublica, Igu\'a 4225, 11400
Montevideo, Uruguay.}
\email{fabadie@cmat.edu.uy}
\urladdr{http://www.cmat.edu.uy/~fabadie/}

\author[Eusebio Gardella]{Eusebio Gardella}
\address{Eusebio Gardella
Mathematisches Institut, Fachbereich Mathematik und Informatik der
Universit\"at M\"unster, Einsteinstrasse 62, 48149 M\"unster, Germany.}
\email{gardella@uni-muenster.de}
\urladdr{www.math.uni-muenster.de/u/gardella/}

\author{Shirly Geffen}
\address{Shirly Geffen
Department of Mathematics, Ben-Gurion University of the Negev, Be’er
Sheva 8410501, Israel.}
\email{shirlyg@post.bgu.ac.il}
\urladdr{http://www.shirlygeffen.com}
\begin{document}

\begin{abstract}
We develop the notion of Rokhlin dimension for partial actions
of finite groups, extending the well-established theory for 
global systems. The partial setting exhibits phenomena that 
cannot be expected for global actions, usually stemming from the 
fact that virtually all averaging arguments for finite group
actions completely break down for partial systems. For example,
fixed point algebra and crossed product are not in general
Morita equivalent, and there is in general no local approximation
of the crossed product $A\rtimes G$ by matrices over $A$.
By using decomposition arguments for partial actions of finite groups, we show that a number of structural properties
are preserved by formation of crossed products, including finite
stable rank, finite nuclear dimension, and absorption of a strongly self-absorbing \ca. Some of our results are new even
in the global case.

We also study the Rokhlin dimension of globalizable actions: while in general it differs from the Rokhlin dimension of its globalization, we show that they agree if the coefficient algebra
is unital. For topological partial actions on spaces of finite
covering dimension, we show that
finiteness of the Rokhlin dimension is equivalent to freeness,
thus providing a large class of examples to which our theory 
applies.
\end{abstract}

\maketitle


\renewcommand*{\thetheoremintro}{\Alph{theoremintro}}
\section{Introduction}

Partial dynamical systems have implicitly been used in  mathematics long before the notion was formalized, 
at least since the study of differential equations. Indeed,
the flow of a differentiable vector field can be naturally
regarded as a partial action of the reals. More precisely, 
given a smooth vector field $v$ on a manifold $M$, 
for $x\in M$ let $\phi_x$ be the unique 
solution to the differential
equation $\phi'(t)=v(\phi(t))$ with initial condition
$\phi_x(0)=x$, and let $I_x\subseteq \mathbb{R}$ 
be the largest (open) neighborhood of $0$ on which $\phi_x$ is defined. 
For $t\in\mathbb{R}$, set 
\[
M_t=\{x\in M\colon t\in I_x\},                           
\]
and let $\sigma_t\colon M_{-t}\to M_t$ be the diffeomorphism
given by
$\sigma_t(x)=\phi_x(t)$ for all $x\in M_{-t}$. 
A crucial property that the collection $
\sigma=\{\sigma_t,M_t\colon t\in
\mathbb{R}\}$ satisfies, is that $\sigma_{s+t}$ extends
$\sigma_s\circ\sigma_t$, in the sense that whenever
$x\in M_{-t}$ and $\sigma_t(x)$ belongs to $M_{-s}$, then
$x$ belongs to $M_{-s-t}$ and $\sigma_{s+t}(x)=\sigma_s(\sigma_t(x))$. In modern language, this condition
asserts that $\sigma$ is a \emph{partial action} 
of $\mathbb{R}$ on $M$.

Partial actions were originally introduced by Exel and
McClanahan in the 1990s, by isolating and abstracting the conditions 
observed in the context described above: 
a partial action of a discrete group $G$ on a topological space $X$ is a 
collection $\{X_g\colon g\in G\}$ of open sets of $X$, and homeomorphisms
$\sigma_g\colon X_{g^{-1}}\to X_g$ such that $\sigma_1=\id_X$ and $\sigma_{gh}$ extends $\sigma_g\circ \sigma_h$ wherever the decomposition is well-defined.
The notion of a \emph{global} (or ordinary) action is obtained by taking $X_g=X$
for all $g\in G$.
We refer the reader to the recent book \cite{Exe_book_2017} for a modern 
treatment of this topic and historical references. 
The study of partial actions (both on topological spaces and on $C^*$-algebras) 
has been very fruitful, and has shed new light on the study of several 
objects. For example, the fact that the solutions of a differential equation
on a \emph{compact} manifold are defined on all $\mathbb{R}$ can be easily
proved in this more abstract setting; see Proposition~2.4 in~\cite{Aba_enveloping_2003}.

A typical example of a partial action is obtained by starting with a global
action $\beta\colon G\to \Aut(B)$, a not necessarily invariant ideal
$A$ in $B$, and setting $A_g=A\cap \beta_g(A)$ with 
$\alpha_g=\beta_g|_{A_{g^{-1}}}$ for all $g\in G$. Partial actions obtained in
this way are called \emph{globalizable}, and the \emph{globalization problem} asks to determine whether a given
partial action is globalizable, and in case it is, describe its 
globalization; see Section~3 of~\cite{Aba_enveloping_2003}. 
As it turns out, not every partial action is globalizable (a necessary and sufficient condition is given in~\cite{Fer_constructions_2018}).
Even when a globalization exists, identifying it is often a challenging task,
and its dynamical properties may differ significantly from those of $\alpha$. 

Given a partial group action $\alpha$ of a group $G$ on a \ca\ $A$,
one can construct its \emph{crossed product} $A\rtimes_\alpha G$;
see Section~I.8 in~\cite{Exe_book_2017}. 
Large families of 
\ca s can be naturally described as partial crossed products, typically 
with commutative \ca s, even in situations where similar descriptions 
do not exist for global crossed products. For example, every unital AF-algebra arises
as the crossed product of a partial homeomorphism of a totally disconnected 
compact space,
while no unital AF-algebra arises as the crossed product of a homeomorphism. 

It is therefore particularly important to develop tools to study partial 
crossed products. There have been a number of advancements in this direction,
for example in what refers to $K$-theory \cite{Exe_circle_1994}, Takai duality \cite{Aba_enveloping_2003}. On the other hand, 
the study of partial actions of \emph{finite} groups remains 
conspicuously underdeveloped. Indeed, and even in the globalizable case, virtually 
all averaging arguments (and their consequences) that are 
standard for global actions, completely break 
down in this setting. The lack of approximate identities that
are compatible with the partial action is also a source of difficulties in this setting. The goal of the present work is to make advances in the study of the 
structure of crossed products by partial actions of finite groups.

In the modern literature in $C^*$-dynamical systems, 
several Rokhlin-type properties play an increasingly central role in the study
of crossed products; see \cite{Izu_finiteI_2004, OsaPhi_crossed_2012, HirWinZac_rokhlin_2015, SzaWuZac_rokhlin_2014, GarHir_strongly_2018}. 
Their wide and fruitful applicability in the global setting make
the extension of this theory to the partial setting worthy of exploration. 

Motivated by \cite{HirWinZac_rokhlin_2015}, we define and 
study the notion of 
\textit{Rokhlin dimension} in the partial setting;
see \autoref{df:Rdim}. The theory that we develop here exhibits phenomena that 
cannot be expected for global actions. Among others, the Rokhlin dimension of a globalizable partial action does not agree with that
of its globalization; see \autoref{eg:RdimGlobDifferent}. A notable exception 
is the unital case:

\begin{thmintro}(\autoref{thm:RdimGlobaliz})
Let $\alpha$ be a globalizable partial action of a finite group
on a \emph{unital} $C^*$-algebra, and let $\beta$ be its globalization. 
Then
\[\dimRok(\alpha)=\dimRok(\beta) \ \ \ \text{ and } \ \ \ \cdimRok(\alpha)=\cdimRok(\beta).\]
\end{thmintro}

Our original motivation was the study of the structure of 
the crossed product, particularly from the point of view of the classification
programme for simple nuclear \ca s; see \cite{EllTom_regularity_2008}. 
As it turns out,
this is technically much more complicated than in the global setting, and 
tackling this problem required us to first develop a decomposition theory for 
partial actions of finite groups into iterated extensions of relatively
simpler systems; see \cite{AbaGarGef_decomposable_2020} and particularly 
Section~6 there. Here, we prove:

\begin{thmintro}(\autoref{thm:CPRdim})
The following properties are inherited by crossed products or fixed point
algebras by partial actions of finite groups with $\dimRok<\I$:
\be\item Having finite nuclear dimension or decomposition rank. For example,
	\[ \dimnuc(A\rtimes_{\alpha} G)\leq (|G|-1)(\dim_{\mathrm{nuc}}(A)+1)(\dim_{\mathrm{Rok}}(\alpha)+1)+\dimnuc(A). \]
\item When $\cdimRok(\alpha)<\I$, having finite stable/real rank.
For example,
\[\mathrm{sr}(A\rtimes_\alpha G)\leq \frac{|G|(\sr(A)+\cdimRok(\alpha)+3)-2}{2}.
\]
\item When $\cdimRok(\alpha)<\I$, absorbing a given strongly self-absorbing \ca.
\ee
\end{thmintro}

For unital partial actions, even more can be said; see \autoref{thm:CPRdimUnital}.
Most remarkably, the UCT is preserved in the commuting towers version.
In this case, we also show that $A^\alpha$ is Morita equivalent to 
$A\rtimes_\alpha G$ (see \autoref{thm:MoritaEqFPtCP}, a fact that 
(rather surprisingly) fails if unitality of $A$ is dropped (see \autoref{eg:NotMEq}).

Our structural results for crossed products are complemented by the fact 
that partial actions with finite Rokhlin dimension are relatively common.
For example, we show that this notion is equivalent to freeness in the 
commutative setting.

\begin{thmintro}(\autoref{thm:FreeFiniteRdim})
Let $\sigma$ be a partial action of a finite group $G$ on a locally compact Hausdorff
space $X$ with $\dim(X)<\I$. Then $\dimRok(\sigma)<\I$ if and only if $\sigma$ is free, in which case we have
\[\dimRok(\sigma)\leq (|G|-1)\dim(X).\]
\end{thmintro}

The result in the global case is implicit 
\cite{HirPhi_rokhlin_2015,
Gar_rokhlin_2017}, and is an easy consequence of the
existence of local cross-sections for the 
quotient map $\pi\colon X\to X/G$. 
However, the proof in the partial setting is considerably
more complicated since for free partial actions there
may not exist local cross-sections for $\pi$. The proof 
in our context is quite involved and occupies most of Section~5. 
The main technical ingredient is the fact (\autoref{prop:ExtRdim}) 
that an extension of topological partial actions with finite 
Rokhlin dimension again has finite Rokhlin dimension.
Roughly speaking, one needs to lift Rokhlin towers from the quotient to the algebra, 
while at the same time respecting the domains of the partial 
action. The fact that the coefficient algebra is commutative seems
to be crucial for this lifting problem to have a solution.

\section{Rokhlin dimension for partial actions of finite groups}
\label{sec: DefRokDim}

In this section we define Rokhlin dimension 
for partial actions. 

\begin{df} \label{df:Rdim}
Let $\alpha=((A_g),(\alpha_g))_{g\in G}$ be a partial action of a finite group $G$ on a $C^*$-algebra $A$.
For $d\in\N$, we say that $\alpha$ has \emph{Rokhlin dimension at most} $d$, and write $\dimRok(\alpha)\leq d$,
if for every $\ep>0$ and every finite subset $F\subseteq A$, there exist positive contractions $f_g^{(j)}\in A_g$, for $g\in G$ and $j=0,\ldots,d$,
satisfying:
\be \item $\left\|\left(\alpha_g(f_h^{(j)}x)-f_{gh}^{(j)}\alpha_g(x)\right)a\right\|<\ep$ for all $g,h\in G$, $j=0,\ldots, d$, $a\in F$ and $x\in A_{g^{-1}}\cap F$;
\item $\left\|f_{g}^{(j)}f_{h}^{(j)}a\right\| < \ep$ for $j=0,\ldots,d$, $g,h \in G$ with $g \neq h$ and $a\in F$;
\item $\left\| \left(\sum_{j=0}^{d}\sum_{g \in G}  f_{g}^{(j)}\right)a-a\right\|<\ep$ for all $a\in F$;
\item $\left\|\left(f_{g}^{(j)}b-bf_g^{(j)}\right)a\right\| < \ep$ for all $j =0,\ldots, d$, $g\in G$, and $a,b\in F$.
\ee
Moreover, we say that $\alpha$ has \emph{Rokhlin dimension with commuting towers at most} $d$, and write $\cdimRok(\alpha)\leq d$,
if for every $\ep>0$ and every finite subset $F\subseteq A$, there exist positive contractions $f_g^{(j)}\in A_g$, for $g\in G$ and $j=0,\ldots,d$,
satisfying conditions (1) through (4) above, in addition to
\be
\item[(5)] $\left\|\left(f_{g}^{(j)}f_{h}^{(k)}-f_{h}^{(k)}f_{g}^{(j)}\right)a\right\| < \ep$ for all $j,k =0,\ldots, d$, $g,h\in G$ and $a\in F$.
\ee
In either case, we call the elements $f_g^{(j)}$ above
\emph{Rokhlin towers} for $(F,\ep)$.

We define the \emph{Rokhlin dimension} of $\alpha$ by
\[\dimRok(\alpha)=\min\{d\in \N\colon \dimRok(\alpha)\leq d\},\]
and define the \emph{Rokhlin dimension with commuting towers} $\cdimRok(\alpha)$ similarly. 
\end{df}

The multiplicative witness $a\in F$ that appears in (1) through (5) also appears in the definition of Rokhlin 
dimension for \emph{global} actions
setting (\cite{HirPhi_rokhlin_2015}), and it can be omitted when $A$ is unital.
On the other hand, the witness $x\in F\cap A_{g^{-1}}$ 
is a conceptually new condition, which cannot be omitted
even if $A$ is unital.

At the early stages of this project, it was unclear whether
one should not instead require the stronger condition
$\|\alpha_g(f_h^{(j)}x)-f_{gh}^{(j)}\alpha_g(x)\|<\ep\|x\|$
for all $g,h\in G$, $j=0,\ldots,d$ and \emph{all} 
$x\in A_{g^{-1}}$. As it turns out, the stronger condition
implies that the given partial action is in fact globalizable,
which suggested that it was not the right notion to consider.

Unlike in the case of global actions,
the different elements within one tower are not 
``interchangeable'', as they tend to have different ``sizes''.
In fact, the positive contraction corresponding to the unit of the group is usually much larger than the others. 
An extreme example of this situation is the following:

\begin{eg}\label{eg:trivialRp}
Let $G$ be a finite group, and let $A$ be a unital \ca. We define the
\emph{trivial partial action} of $G$ on $A$ by 
setting $A_g=\{0\}$ for $g\in G\setminus\{1\}$. 
This partial action has
Rokhlin dimension zero, with Rokhlin towers given by 
$f_1=1_A$ and $f_g=0$ for $g\in G\setminus\{1\}$. Note that
$A\rtimes_\alpha G=A$. Moreover, this is the only partial action
of $G$ on $A$ with the 1-decomposition property 
(\autoref{df:nIntProp}); see Example~2.6 in~\cite{AbaGarGef_decomposable_2020}. 
\end{eg}

Next, we show that the condition (1) in
\autoref{df:Rdim} can be strengthened in 
the case of \emph{unital} partial actions (that is, partial actions whose domains are unital). 

\begin{rem}\label{rem:UnitsIdeals}
Let $\alpha$ be a unital partial action of a finite group $G$ on a \ca\ $A$, with units $1_g\in A_g$, for $g\in G$.
Then 	
$1_{gh}1_g=\alpha_g(1_h1_{g^{-1}})$,
for all $g,h\in G$.
\end{rem}


\begin{prop} \label{prop:RdimReductionUnital}
Adopt the notation from \autoref{df:Rdim}, and suppose that $\alpha$ is a unital partial action. For $g\in G$, denote
by $1_g$ the unit of $A_g$. Then 
condition (1) in \autoref{df:Rdim} can be replaced by
\be
\item[(1')] $\alpha_g(f_h^{(j)}1_{g^{-1}})=f_{gh}^{(j)}1_g$ for all $g,h\in G$, and for all $j=0,\ldots, d$.
\ee
\end{prop}
\begin{proof}
We prove the proposition for $\dimRok$, since the proof for $\cdimRok$ is analogous. 
Using the identity $x=x1_{g^{-1}}$ for all $x\in A_{g^{-1}}$ and $g\in G$, one easily shows that 
(1') implies the following identity 
for all $g,h\in G$, $j=0,\ldots,d$ and all $x\in A_{g^{-1}}$
\[\alpha_g(f_h^{(j)}x)=f_{gh}^{(j)}\alpha_g(x).\]
This identity clearly implies (1). Conversely,
let $\ep>0$ and a finite subset $F \subseteq A$ be given.
Without loss of generality, we assume that $F$ consists of contractions and that 
$\{1_g\colon g\in G\}\subseteq F$. Set 
$\ep_0=\frac{\ep}{|G|(d+1)+1}$ and 
find Rokhlin towers $f_g^{(j)}\in A_g$, for $g\in G$ and $j=0,\ldots,d$, satisfying (1),(2), (3), and (4) in
\autoref{df:Rdim} for $(F,\ep_0)$.
Define positive contractions $\widetilde{f}_g^{(j)}\in A_g$, for all $g\in G$ and $j=0,\ldots,d$, by $\widetilde{f}_g^{(j)}=\alpha_g(f_1^{(j)}1_{g^{-1}})$. Since
$1_{g^{-1}}$ belongs to $F$, condition
(1) for $f_1^{(j)}$ 
gives
\[\label{eqn:fgtildeclosef_g}\tag{2.1}
\|\widetilde{f}_g^{(j)}- f_g^{(j)}\|=\|\alpha_g(f_1^{(j)}1_{g^{-1}})-f_g^{(j)}1_g\|<\ep_0
\]
for all $g\in G$ and all $j=0,\ldots,d$. 
We claim that these elements satisfy the conditions in \autoref{df:Rdim} with (1) replaced by (1').

We begin with (1'). For $g,h\in G$ and $j=0,\ldots,d$, we have
\begin{align*}
 \alpha_g(\widetilde{f}_h^{(j)}1_{g^{-1}})&=\alpha_g(\alpha_h(f_1^{(j)}1_{h^{-1}})1_{g^{-1}})\\
 &=\alpha_g(\alpha_h(f_1^{(j)}1_{h^{-1}}1_{{(gh)}^{-1}}))\\
 &=\alpha_{gh}(f_1^{(j)}1_{{(gh)}^{-1}})1_g
 =\widetilde{f}_{gh}^{(j)}1_g.
\end{align*}
Finally, conditions (2), (3) and (4) for the 
$\widetilde{f}_g^{(j)}$ follow by combining 
$(\ref{eqn:fgtildeclosef_g})$ with conditions (2), (3) and (4) for $f_g^{(j)}$. We omit the details.
\end{proof}

\begin{rem} In the context of \autoref{prop:RdimReductionUnital}, 
one can show that condition (2) can be replaced by
\be\item[(2')] $f_g^{(j)}f_h^{(j)}=0$ for all $g,h\in G$ with $g\neq h$ and for all $j=0,\ldots,d$.
\ee
Since we do not need this, we omit its proof. 
We stress the fact that it is in general \emph{not} 
possible to replace (1) and (2) 
\emph{simultaneously} with 
(1') and (2'), since the argument used to get (1') 
from (1) does not preserve (2'), and
vice-versa. 
\end{rem}

We close this section by proving that finite Rokhlin dimension behaves well with respect to restriction to invariant ideals and passage to equivariant quotients. 

\begin{prop}\label{prop:PermProp}
Let $A$ be a \ca, let $G$ be a finite group, and let $\alpha$ be a partial
action of $G$ on $A$. Let $I$ be a $G$-invariant ideal of $A$. We denote by $\alpha|_{I}$ and $\overline{\alpha}$ the induced partial actions of $G$ on $I$ and $A/I$, respectively. Then
\[\dimRok(\alpha|_{I})\leq \dimRok(\alpha)\ \ 
\mbox{ and } \ \ \dimRok(\overline{\alpha})\leq \dimRok(\alpha).\]
Similar estimates hold for $\cdimRok$. \end{prop}
\begin{proof}
We prove the results for $\dimRok$, since the case of $\cdimRok$ is similar.
We assume from now on that $d=\dimRok(\alpha)<\I$, otherwise there is nothing to prove.

We prove $\dimRok(\alpha|_{I})\leq \dimRok(\alpha)$ first. Let $\ep>0$ and let a finite subset $F\subseteq I$ be given.
Without loss of generality, we assume that $F$ contains
only contractions. 
Set $\ep_0=\frac{\ep}{5|G|(d+1)+2}$.
Apply \autoref{df:Rdim} to find Rokhlin towers 
$f_g^{(j)}\in A_g$ with $g\in G$ and $j=0,\ldots,d$, 
for $(F, \ep_0)$. By making a small renormalization,
we may assume that $\sum_{g\in G}\sum_{j=0}^df_g^{(j)}$ is 
a contraction.
By considering an approximate identity of 
$I$ which is quasi-central in $A$, find $e\in I$ satisfying 
\[\left\| e^{\frac{1}{2}}f_g^{(j)}-f_g^{(j)}e^{\frac{1}{2}}\right\|<\ep_0, \ \ 
 \left\|be-b\right\|<\ep_0, \ \text{ and } \ \left\|eb-b\right\|<\ep_0  
\]
for all $g\in G$, all $j=0,\ldots,d$, and all 
$b\in \bigcup\limits_{g\in G}\alpha_g(F\cap A_{g^{-1}}).$

For $g\in G$ and $j=0,\ldots,d$, set 
$\widetilde{f}_g^{(j)}=e^{\frac{1}{2}}f_g^{(j)}e^{\frac{1}{2}}$. 
We claim that $\{\widetilde{f}_g^{(j)}\colon g\in G, j=0,\ldots,d\}$ are Rokhlin towers with respect to $(F, \ep)$
for $\alpha|_I$. Note that $\widetilde{f}_g^{(j)}$
belongs to $A_g\cap I=I_g$.
In order to check (1), let $g\in G$, let 
$x\in F\cap I_{g^{-1}}=F\cap A_{g^{-1}}$, let $a\in F$, and let 
$j=0,\ldots,d$. Then
\begin{align*}
 \alpha_g(\widetilde{f}_h^{(j)}x)a&= \alpha_g(e^{\frac{1}{2}}f_h^{(j)}e^{\frac{1}{2}}x)a
\approx_{\ep_0} \alpha_g(f_h^{(j)}ex)a\\
 &\approx_{\ep_0} \alpha_g(f_h^{(j)}x)a
\approx_{\ep_0} f_{gh}^{(j)}\alpha_g(x)a\\
 &\approx_{\ep_0} f_{gh}^{(j)}e\alpha_g(x)a
 \approx_{\ep_0} e^{\frac{1}{2}}f_{gh}^{(j)}e^{\frac{1}{2}}\alpha_g(x)a
 = \widetilde{f}_{gh}^{(j)}\alpha_g(x)a.
\end{align*}
Thus $\|\alpha_g(\widetilde{f}_h^{(j)}x)a-
\widetilde{f}_{gh}^{(j)}\alpha_g(x)a\|< 5\ep_0<\ep$,
as desired.
Condition (2) is easily checked, and is left to the reader.
For (3), let $a\in F$ be given. Then
\begin{align*}
\sum\limits_{j=0}^{d}\sum\limits_{g\in G} \widetilde{f}_g^{(j)}a 
&=\sum\limits_{j=0}^{d} \sum\limits_{g\in G} e^{\frac{1}{2}}f_g^{(j)}e^{\frac{1}{2}}a\\
&\approx_{|G|(d+1)\ep_0} \sum\limits_{j=0}^{d} \sum\limits_{g\in G} f_g^{(j)}ea\\
&\approx_{\ep_0} \sum\limits_{j=0}^{d} \sum\limits_{g\in G} f_g^{(j)}a\approx_{\ep_0}a,
\end{align*}
where in the second to last step we used the fact that 
$\sum_{g\in G}\sum_{j=0}^df_g^{(j)}$ is a contraction.
Finally, to check (4), let $g\in G, j=0,\ldots,d$ and $a\in F$. 
Then
\begin{align*}
\widetilde{f}_g^{(j)}a=e^{\frac{1}{2}}f_g^{(j)}e^{\frac{1}{2}}a\approx_{2\ep_0} f_g^{(j)}a \approx_{\ep_0} xf_g^{(j)}
\approx_{2\ep_0}xe^{\frac{1}{2}}f_g^{(j)}e^{\frac{1}{2}}= a\widetilde{f}_g^{(j)}.
\end{align*}

We turn to the inequality $\dimRok(\overline{\alpha})\leq d=\dimRok(\alpha)$. Write $\pi\colon A\to A/I$ for the 
canonical equivariant map. 
Let $\overline{F}\subseteq A/I$ be a finite set and $\ep>0$. Let
$F\subseteq A$ be any finite set satisfying
$\pi(F)=\overline{F}$, and apply \autoref{df:Rdim} 
for $\alpha$ to find
Rokhlin towers $f_g^{(j)}\in A_g$, with $g\in G$ and $j=0,\ldots,d$, for $(F,\ep)$. 
It is then immediate to check that
the positive contractions $\overline{f}_g^{(j)}=\pi(f_g^{(j)})$ witness the fact
that $\dimRok(\overline{\alpha})\leq d$, as desired.
\end{proof}


\section{Rokhlin dimension and globalization} \label{sec: RokDimGlobRes}
The basic example of a partial action is obtained by 
starting with a global action $\beta\colon G\to \Aut(B)$
and a (not necessarily invariant) ideal $A$ in $B$, and then 
setting $A_g=A\cap \beta_g(A)$ 
and $\alpha_g=\beta_g|_{A_{g^{-1}}}$ for all $g\in G$.
Actions of this form are called \emph{globalizable}, since
they are induced by a global action. 
Here is the precise definition:

\begin{df}\label{df:enveloping}
Let $G$ be a finite group, let $A$ be a \ca, and let 
$\alpha$ be a partial action of $G$ on 
$A$. A triple $(B,\beta,\iota)$ consisting of a \ca\ $B$, a global action
$\beta\colon G\to\Aut(B)$ and an embedding 
$\iota\colon A\to B$, is said to 
be an \emph{enveloping action} for $\alpha$ if the following conditions are
satisfied:
\be\item[(a)] $A$ (identified with $\iota(A)$) is an ideal in $B$;
\item[(b)] $A_g=A\cap \beta_g(A)$ for all $g\in G$;
\item[(c)] $\alpha_g(a)=\beta_g(a)$ for all $a\in A_{g^{-1}}$ and all $g\in G$;
\item[(d)] $B=\overline{\mathrm{span}}\{\beta_g(a)\colon a\in A, g\in G \}$.\ee
(If such a dynamical system $(B,\beta)$ exists, then it is
unique up to
an equivariant isomorphism extending the identity on $A$ by
Theorem~3.8 in~\cite{Aba_enveloping_2003}.)
We say that $\alpha$ is \emph{globalizable} if it has an 
enveloping action.
\end{df}

Not every partial action is globalizable, and even when it is, identifying its
enveloping action may turn out to be challenging. 
Since there is a vast amount of literature concerning global
actions with finite Rokhlin dimension, it would be very useful
if one could relate the Rokhlin dimension of a (globalizable) 
partial action to its globalization.
Unfortunately, these values do not necessarily agree:

\begin{eg}\label{eg:RdimGlobDifferent}
Set $X=S^1\setminus\{1\}$ and $U=X\setminus\{-1\}$.
Let $\sigma\in\mathrm{Homeo}(U)$ be given by $\sigma(x)=-x$
for all $x\in U$. 
Denote by $\gamma$ the partial action of $\Z_2=\{-1,1\}$ on $C_0(X)$ 
determined by $\sigma$. Then $\gamma$ is globalizable, with
globalization $\widetilde{\gamma}\colon \Z_2\to \Aut(C(S^1))$
induced by multiplication by -1. 

Let $\delta\colon \Z_2\to\Aut(\C\oplus \C)$ be the flip
action. Set $\alpha=\gamma\otimes\delta$, which is globalizable
with globalization given by 
$(B,\beta)=(C(S^1)\oplus C(S^1),\widetilde{\gamma}\otimes\delta)$. 
It is clear that $\dimRok(\beta)=0$, since we may take 
$p_{-1}=(1,0)$ and $p_1=(0,1)$ in $B$.

We claim that $\dimRok(\alpha)\neq 0$. 
From now on, we 
identify $X$ with $(0,2)$ and $U$ with $(0,1)\cup (1,2)$.
Arguing by contradiction, set $\ep=3/16$ and 
take $I_{1}=[\ep,2-\ep]$ and 
$I_{-1}=[\ep,1-\ep]\cup [1+\ep,2-\ep]$, let 
$a\in C_0(X)$ be constant equal to 1 on $I_1$, 
and linear otherwise, and let $b=\sigma(b)\in C_0(U)=C_0(X)_{-1}$ be constant 
equal to 1 on $I_{-1}$ and linear otherwise.
Let $f_{-1}=(\xi_{-1},\eta_{-1})\in C_0(U)\oplus C_0(U)$ and $f_1=(\xi_1,\eta_1)\in C_0(X)\oplus C_0(X)$
be a Rokhlin tower for $\alpha$ with respect to $F=\{(a,a),(b,b)\}$ and $\ep$. Then we have
\bi\item[(a)] $\big|\xi_{-1}(\sigma(x))b(x)-\eta_1(x)b(x)\big|<\ep$ and 
$\big|\eta_{-1}(\sigma(x))b(x)-\xi_1(x)b(x)\big|<\ep$;
\item[(b)] $\xi_{-1}(x)\xi_1(x)a(x)<\ep$ and 
$\eta_{-1}(x)\eta_1(x)a(x)<\ep$;
\item[(c)] $(1-\xi_{-1}(x)-\xi_1(x))a(x)<\ep$ and $(1-\eta_{-1}(x)-\eta_1(x))a(x)<\ep$\ei
for all $x\in X$. 
Upon making a small renormalization, we may assume that
\[\label{eqn:3.1}\tag{3.1}\xi_{-1}(x)+\xi_1(x)=1=\eta_{-1}(x)+\eta_1(x)\]
for all $x\in I_1$. In particular $\xi_1(1)=1$.
Fix $x\in I_1$. Substituting (\ref{eqn:3.1}) into (b), we get 
\[(1-\xi_1(x))\xi_1(x)<\ep,\]
which yields either $\xi_1(x)> 3/4$ or $\xi_1(x)<1/4$. 
Since $\xi_1$ is continuous, we must have
either $\xi_1(I_{1})\subseteq  (3/4,1]$ or 
$\xi_1(I_{1})\subseteq [0,1/4)$, and since $\xi_1(1)=1$,
it must be $\xi_1(I_{1})\subseteq  (3/4,1]$ and thus 
\[\label{eqn:3.2}\tag{3.2}  
\xi_{-1}(I_{1})\subseteq [0,1/4).\]
An identical argument shows that
\[\label{eqn:3.3}\tag{3.3}\eta_1(I_{1})\subseteq  (3/4,1].\]
Taking now $x\in I_{-1}\subseteq I_1$, and noting that $b(x)=1$
and $\sigma(x)\in I_{-1}$ as well,
we get
\[\frac{3}{16}=\ep\stackrel{\mathrm{(a)}}{>}\big|\xi_{-1}(\sigma(x))b(x)-\eta_1(x)b(x)\big|\stackrel{\ref{eqn:3.2},\ref{eqn:3.3}}{>} \frac{3}{4}-\frac{1}{4},\]
which is a contradiction. We conclude that $\dimRok(\alpha)>0$.

We point out that one can construct towers that witness 
the fact that 
$\dimRok(\alpha)\leq 1$, thus allowing us to conclude that
$\dimRok(\alpha)=1$. However, we do not need this, so we omit it.
\end{eg}

In contrast with the previous example, we will show in
\autoref{thm:RdimGlobaliz} that for globalizable
partial actions which act on \emph{unital} \ca s, 
their Rokhlin dimension (with or without 
commuting towers) agrees with that of its globalization. 
The result is by no means obvious and 
its proof is quite technical. The following lemma,
which deals exclusively with global actions,
represents the first step in proving 
the inequality $\dimRok(\alpha)\leq \dimRok(\beta)$.

\begin{lma} \label{lemma: Restriction intermediate lemma}
Let $\beta\colon G\to \Aut(B)$ be an action of a finite
group $G$ on a \uca\ $B$. Let $A$ be a unital ideal in 
$B$
satisfying $B=\overline{\mathrm{span}}\{\beta_g(a)\colon a\in A, g\in G \}$. If $d=\dimRok(\beta)$ is finite, 
then there exist Rokhlin towers $f_g^{(j)}\in B$
for $\beta$, satisfying condition (1') 
in \autoref{prop:RdimReductionUnital}, and such that
$f_1^{(0)},\ldots,f_1^{(d)}$ belong to $A$. 
A similar statement holds when $\cdimRok(\beta)<\I$.
\end{lma}

\begin{proof} 
We divide the proof into claims:

\textbf{Claim 1:} \emph{there exist projections
$p_g\in A$, for $g\in G$, which are central in $B$
and satisfy $1_B=\sum\limits_{g\in G}\beta_g(p_g)$.}

Set $n=|G|$ and fix an enumeration 
$G=\{1_G=g_1,g_2,\ldots,g_n\}$. Since 
$B=\sum_{g\in G}\beta_g(A)$,
we have $1_B\leq \sum\limits_{g\in G}\beta_g(1_A)$. 
Note that the projections $\beta_g(1_A)$ are central, and
therefore together with $1_B$ generate a commutative \ca. 
In the rest of this claim, we regard $1_B$ and $\beta_g(1_A)$,
for $g\in G$, as $\{0,1\}$-valued functions on some
compact Hausdorff space, and identify them with their supports.
In particular, the support of $1_B$ is equal to the union of the 
supports of $\beta_g(1_A)$, for $g\in G$.
By successively removing the double intersections, adding
the triple intersections, and similarly for higher 
degrees, we can write $1_B$ as follows:
\begin{align*}\label{eqn:Decomp1B}\tag{3.4}
1_B=&\sum\limits_{i=1}^{n}\beta_{g_i}(1_A)-\sum\limits_{1\leq i<j\leq n}\beta_{g_i}(1_A)\beta_{g_j}(1_A)
\\
& + \sum_{1\leq i<j<k\leq n}\beta_{g_i}(1_A)\beta_{g_j}(1_A)\beta_{g_k}(1_A) +\cdots \\
& +(-1)^{n-1}\beta_{g_1}(1_A)\beta_{g_2}(1_A)\cdots\beta_{g_n}(1_A).\end{align*}

Observe that every element appearing in (\ref{eqn:Decomp1B}) is central in $B$.
We now proceed to write the unit of $B$ as a sum of orthogonal projections, in the following manner.
The first summand consists of 
all the elements appearing in (\ref{eqn:Decomp1B}) 
that have $\beta_{g_1}(1_A)$ as a factor.
The second summand consists of all the elements 
appearing in (\ref{eqn:Decomp1B}) 
that have $\beta_{g_2}(1_A)$ as a factor but not $\beta_{g_1}(1_A)$. Continue
inductively, and note that the process finishes
after $n$ steps; indeed, 
the only element left at the $n$-th step is 
$\beta_{g_n}(1_A)$.

For $1\leq k\leq n$, the $k$-th summand is
\begin{align*}
P_k= \beta_{g_k}(1_A)&-\sum\limits_{k<j\leq n}\beta_{g_k}(1_A)\beta_{g_j}(1_A)+\sum\limits_{k<i<j\leq n}\beta_{g_k}(1_A)\beta_{g_i}(1_A)\beta_{g_j}(1_A)\\
&+\cdots +(-1)^{n-k}\beta_{g_k}(1_A)\beta_{g_{k+1}}(1_A)\cdots \beta_{g_n}(1_A),
\end{align*}
and we have $1_B=P_1+\cdots+P_n$. 
We want to see that $P_k$ is a projection. Set 
\[Q_{k}=\sum\limits_{k<j\leq n}\beta_{g_j}(1_A)-\sum\limits_{k<i<j\leq n}\beta_{g_i}(1_A)\beta_{g_j}(1_A)+\cdots + (-1)^{n-k+1}\prod_{j>k} \beta_{g_j}(1_A).\]
Then $Q_{k}$ is the unit of the ideal 
$\sum\limits_{j>k}\beta_{g_j}(A)$, and therefore 
$Q_{k}$ is a projection. Moreover, an easy computation
shows that
$P_k=\beta_{g_k}(1_A)(1_B-Q_{k})$, and thus $P_k$ is 
also a projection. Note that $P_kP_\ell=0$ if $k\neq \ell$.

For $k=1,\ldots,n$, set $p_k=\beta_{g_k}^{-1}(P_k)$,
which can be written as
\[p_k=1_A-\sum\limits_{k<j\leq n}1_A\beta_{{g_k}^{-1}g_j}(1_A)+\cdots +(-1)^{(n-k)}1_A\beta_{{g_k}^{-1}g_{k+1}}(1_A)\cdots \beta_{{g_k}^{-1}{g_n}}(1_A).\]
Then $p_k$ is a central projection, and it belongs to $A$
because $1_A$ is a factor in each of its summands.
Since $1_B=\sum_{k=1}^n \beta_{g_k}(p_k)$, 
this concludes the proof of Claim~1. 

\vspace{0.3cm}

Let $\ep>0$ and let $F\subseteq B$ be finite.
Without loss of generality, $F$ is $\beta$-invariant,
contains $1_B$, 
and consists of contractions. 
Set $\ep_0=\frac{\ep}{|G|^2(d+2)}$, and 
let $f_g^{(j)}\in B$, for $g\in G$ and $j=0,\ldots,d$,
be Rokhlin towers for $(F,\ep_0)$. 
Using \autoref{prop:RdimReductionUnital} for the equality,
and by replacing $f_g^{(j)}$ with $\frac{1}{1+\ep}f_g^{(j)}$ for 
the inequality,
we may assume
\[\label{eqn:Simplifications}\tag{3.5}
\beta_g(f_h^{(j)})=f_{gh}^{(j)} \ \ \mbox{ and } \ \ \sum\limits_{g\in G}\sum_{j=0}^df_g^{(j)}\leq 1\]
for all $g,h\in G$ and all $j=0,\ldots,d$.

\textbf{Claim 2:} \emph{there are positive elements $x_g^{(j)}\in A$, for
$g\in G$ and $j=0,\ldots,d$, with
\be\item[(2.a)] $f_1^{(j)}=\sum\limits_{g\in G}\beta_g(x_g^{(j)})$;
\item[(2.b)] $\sum\limits_{g\in G} x_g^{(j)}\in A$ is a positive contraction;
\item[(2.c)] $\beta_h(x_g^{(j)})b\approx_{\ep_0}b\beta_h(x_g^{(j)})$ for all $b\in F$, all $g,h\in G$ and all $j=0,\ldots,d$; and 
\item[(2.d)] $\left\|x_h^{(j)}\beta_g(x_t^{(j)})\right\|<\ep_0$ for all $j=0,\ldots,d$ and all $g,h,t\in G$ with $g\neq 1$.\ee}

Using Claim~1, fix projections $p_g\in A$, for $g\in G$, which are central in $B$ and satisfy $1_B=\sum_{g\in G}\beta_g(p_g)$. 
For $j=0,\ldots,d$, multiply both sides of the identity by $f_1^{(j)}$ to get:
\[f_1^{(j)}=\sum\limits_{g\in G}\beta_{g}(\beta_{{g}^{-1}}(f_1^{(j)})p_g)=\sum\limits_{g\in G}\beta_{g}(f_{{g}^{-1}}^{(j)}p_g).\]
Set $x_{g}^{(j)}=f_{{g}^{-1}}^{(j)}p_g$ for all 
$g\in G$ and $j=0,\ldots,d$. Since $p_g$ is central in $B$
and belongs to $A$, it is clear that $x_g^{(j)}$ is a positive element in $A$. Condition (2.a) is
satisfied by construction.
Using centrality of $p_g$ at the second step, 
we get
\[\sum\limits_{g\in G}x_{g}^{(j)}=\sum\limits_{g\in G}f_{{g}^{-1}}^{(j)}p_g\leq\sum\limits_{g\in G}f_{g}^{(j)}\leq\sum\limits_{j=0}^{d}\sum\limits_{g\in G}f_{g}^{(j)}
\stackrel{(\ref{eqn:Simplifications})}{\leq}
1,\]
and thus $\sum\limits_{g\in G}x_{g}^{(j)}$ is a positive contraction, verifying (2.b). 
In order to check (2.c), let $g,h\in G$ and $j=0,\ldots,d$.
Since $\beta_h(x_g^{(j)})=f_{hg^{-1}}^{(j)}\beta_h(p_g)$
and $p_g$ is central, it follows that 
\[\left\|\beta_h(x_g^{(j)})b-b\beta_h(x_g^{(j)})\right\|\leq
 \left\|f_{hg^{-1}}^{(j)}b-bf_{hg^{-1}}^{(j)}\right\|<\ep_0
\]
for all $b\in F$. We turn to (2.d). 
Given $j=0,\ldots,d$ and 
$g,h,t\in G$ with $g\neq 1$, we have
\[
x_h^{(j)}\beta_g(x_t^{(j)})=
f_{{h}^{-1}}^{(j)}p_{h}f_{gt^{-1}}^{(j)}\beta_g(p_{t})=
f_{{h}^{-1}}^{(j)}f_{gt^{-1}}^{(j)}\beta_g(p_{t})p_{h}.\]

If $gt^{-1}\neq h^{-1}$ then 
$\|f_{{h}^{-1}}^{(j)}f_{gt^{-1}}^{(j)}\|<\ep_0$
and hence $\|x_h^{(j)}\beta_g(x_t^{(j)})\|<\ep_0$.
Otherwise, we have
$gt^{-1}= h^{-1}$ and thus $g=h^{-1}t$. In particular, 
$h\neq t$. Then 
\[\left\|\beta_g(p_{t})p_{h}\right\|=\left\|\beta_{h^{-1}t}(p_k)p_h\right\|=\left\|\beta_t(p_t)\beta_h(p_h)\right\|=0,\] 
since $\beta_t(p_t)$ is orthogonal to $\beta_h(p_h)$ by Claim~1.
It follows that $x_h^{(j)}\beta_g(x_t^{(j)})=0$, and 
thus (2.d) is satisfied. This proves Claim~2.

\vspace{0.3cm}

Let $x_g^{(j)}\in A$, for $g\in G$ and $j=0,\ldots,d$,
be positive elements satisfying the conclusion of Claim~2.
For $j=0,\ldots,d$, set $a_1^{(j)}=\sum\limits_{g\in G}x_g^{(j)}\in A$. For $g\in G$, we set $a_g^{(j)}=\beta_g(a_1^{(j)})$. 
Then $\beta_g(a_h^{(j)})=a_{gh}^{(j)}$ for all 
$g,h\in G$ and $j=0,\ldots,d$. In particular, condition
(1) in \autoref{df:Rdim} is satisfied. To check condition
(2), let $j=0,\ldots,d$ and $g,h\in G$ with $g\neq h$. 
Then
\[\|a_g^{(j)}a_h^{(j)}\|=\left\|a_1^{(j)}\beta_{g^{-1}h}(a_1^{(j)})\right\|=\left\|\sum\limits_{s,t\in G}x_s^{(j)}\beta_{g^{-1}h}(x_t^{(j)})\right\|\stackrel{\mathrm{(2.d)}}{\leq} |G|^2\ep_0<\ep.\]
Moreover, condition (3) follows from the following identity:
\begin{align*}
\sum\limits_{g\in G}\sum\limits_{j=0}^{d}a_g^{(j)}&=\sum\limits_{g\in G}\sum\limits_{j=0}^{d}\beta_g\left(\sum\limits_{h\in G} x_h^{(j)}\right)=\sum\limits_{g,h\in G}\sum\limits_{j=0}^{d}\beta_{g}(x_h^{(j)}) \\
&=\sum\limits_{g,h\in G}\sum\limits_{j=0}^{d}\beta_{gh}(x_h^{(j)})
=\sum\limits_{g\in G}\sum\limits_{j=0}^{d}\beta_g\left(\sum\limits_{h\in G} \beta_h(x_h^{(j)})\right)\\
&=
\sum\limits_{g\in G}\sum\limits_{j=0}^{d}\beta_g(f_1^{(j)})
=\sum\limits_{g\in G}\sum\limits_{j=0}^{d}f_g^{(j)}.
\end{align*}
Let $b\in F$, $g\in G$ and $j=0,\ldots,d$. In order to
check condition (4) in \autoref{df:Rdim}, and 
since $a_g^{(j)}=\beta_g(a_1^{(j)})$ and 
$F$ is $\beta$-invariant, it suffices to take $g=1$. 
In this case, we have
\[\|a_1^{(j)}b-ba_1^{(j)}\|\leq \sum_{g\in G}\|x_g^{(j)}b-bx_g^{(j)}\|\stackrel{\mathrm{(2.c)}}{\leq} |G|\ep_0<\ep.\]
This proves the first part of the lemma. 

Assume now that 
$\cdimRok(\beta)=d<\infty$, and choose the Rokhlin towers as above to moreover satisfy condition (5) in 
\autoref{df:Rdim}.
For $g,h\in G$ and $j,k=0,\ldots,d$, we get 
\begin{align*}a_g^{(j)}a_h^{(k)} &= \beta_g(a_1^{(j)})\beta_h(a_1^{(k)})=\sum_{t,s\in G} \beta_g(x_t^{(j)})\beta_h(x_s^{(k)})\\
&=\sum_{t,s\in G} \beta_g(f_{t^{-1}}^{(j)}p_t)\beta_h(f_{s^{-1}}^{(k)}p_s)
=\sum_{t,s\in G} f_{gt^{-1}}^{(j)}\beta_g(p_t)f_{hs^{-1}}^{(k)}\beta_h(p_s)\\
&\approx_{|G|^2\ep_0}\sum_{t,s\in G} f_{hs^{-1}}^{(k)}\beta_h(p_s)f_{gt^{-1}}^{(j)}\beta_g(p_t)
=a_h^{(k)}a_g^{(j)}.\qedhere\end{align*}
\end{proof}

We are now ready to prove the main result of this
section: the Rokhlin dimension of a globalizable partial action on a unital 
\ca\ equals the Rokhlin dimension of its globalization. In
particular, we obtain a large family of examples of partial
actions with finite Rokhlin dimension. 

\begin{thm}\label{thm:RdimGlobaliz}
Let $\alpha$ be a globalizable partial action of a finite group $G$ on a unital $C^*$-algebra $A$, and let $\beta\colon G\to\Aut(B)$
denote its globalization.
Then
\[\dimRok(\alpha)=\dimRok(\beta) \ \ \mbox{ and } \ \ \cdimRok(\alpha)=\cdimRok(\beta).\]
\end{thm}

\begin{proof}
The proof for $\dimRok$ and $\cdimRok$ are very similar, 
so we provide full details for $\dimRok$ and indicate
how one modifies the proof to obtain the result for 
$\cdimRok$. 
We divide the proof into two parts, namely the inequalities
$\dimRok(\alpha)\leq \dimRok(\beta)$ and 
$\dimRok(\alpha)\geq\dimRok(\beta)$. 
Since $A$ is unital and $\alpha$ is globalizable, it 
follows that $A_g$ is unital for all $g\in G$, with 
unit given by $1_g=1_A\beta_g(1_A)$.

To show that $\dimRok(\alpha)\leq \dimRok(\beta)$, it suffices
to assume that $d=\dimRok(\beta)$ is finite. 
Let $\ep>0$ and $F\subseteq A$ be given.
Using \autoref{lemma: Restriction intermediate lemma},
let $\widetilde{f}_g^{(j)}\in B$, for $g\in G$ and 
$j=0,\ldots,d$, such that 
\be\item[(a)] $\beta_g(\widetilde{f}_h^{(j)})=\widetilde{f}_{gh}^{(j)}$ for all $g,h\in G$ and $j=0,\ldots,d$;
\item[(b)] $\|\widetilde{f}_g^{(j)}\widetilde{f}_h^{(j)}\|<\ep$ for all $j=0,\ldots,d$ and all $g,h\in G$ with $g\neq h$;
\item[(c)] $\Big\| 1_B-\sum\limits_{g\in G}\sum\limits_{j=0}^d \widetilde{f}_g^{(j)}\Big\|\leq \ep$.
\item[(d)] $\|\widetilde{f}_g^{(j)}b-b\widetilde{f}_g^{(j)}\|<\ep$ for all $g\in G$, all $j=0,\ldots,d$ and all $b\in F$.\ee

Set $f_g^{(j)}=\widetilde{f}_g^{(j)}1_g\in A_g$ for all $g\in G$ and all $j=0,\ldots,d$. We claim that these are Rokhlin towers for $\alpha$ with respect to $(F,\ep)$. 
For $g,h\in G$ and $j=0,\ldots,d$, 
we use at the second step that $\beta_g|_{A_{g^{-1}}}=\alpha_g$ to get
\[\alpha_g(f_h^{(j)}1_{g^{-1}})=\alpha_g(\widetilde{f}_h^{(j)}1_h1_{g^{-1}})=\beta_g(\widetilde{f}_h^{(j)})1_{gh}1_g\stackrel{(a)}{=}\widetilde{f}_{gh}^{(j)}1_{gh}1_g=f_{gh}^{(j)}1_g,\]
thus verifying condition (1) in~\autoref{df:Rdim}.
Moreover, 
$\big\|f_g^{(j)}f_h^{(j)}\big\| \leq \big\|\widetilde{f}_g^{(j)}\widetilde{f}_h^{(j)}\big\|$
and hence condition (2) is also satisfied by (b) above. 
Since
\begin{align*}
\sum\limits_{g\in G}\sum\limits_{j=0}^{d}f_g^{(j)}&=\sum\limits_{g\in G}\sum\limits_{j=0}^{d}\beta_g(\widetilde{f}_1^{(j)})1_g
=\sum\limits_{g\in G}\sum\limits_{j=0}^{d}\beta_g(\widetilde{f}_1^{(j)}1_A)\beta_g(1_A)1_A=\sum\limits_{g\in G}\sum\limits_{j=0}^{d}\widetilde{f}_g^{(j)}1_A,
		\end{align*}
it follows from (c) that condition (3) is also satisfied. Finally,
given $b\in F$, $g\in G$ and $j=0,\ldots,d$, we have
\[\left\|f_g^{(j)}b-bf_g^{(j)}\right\|=\left\|\widetilde{f}_g^{(j)}b1_g-b\widetilde{f}_g^{(j)}1_g\right\|\leq \left\|\widetilde{f}_g^{(j)}b-b\widetilde{f}_g^{(j)}\right\|\stackrel{(d)}{=}\ep,\]
establishing condition (4). It follows that 
$\dimRok(\alpha)\leq d$, as desired. 
Note that 
\[\|f_g^{(j)}f_h^{(k)}-f_h^{(k)}f_g^{(j)}\|\leq 
\|\widetilde{f}_g^{(j)}\widetilde{f}_h^{(k)}-\widetilde{f}_h^{(k)}\widetilde{f}_g^{(j)}\|\]
for all $g,h\in G$ and $j,k=0,\ldots,d$. Thus, if
$\cdimRok(\beta)\leq d$ and the Rokhlin towers 
$\widetilde{f}_g^{(j)}$ for $\beta$ also satisfy condition
(5) in~\autoref{df:Rdim}, then the Rokhlin towers 
$f_g^{(j)}$ for $\alpha$ also satisfy (5) and hence 
$\cdimRok(\alpha)\leq d$.
\vspace{0.3cm}

We turn to the inequality $\dimRok(\beta)\leq \dimRok(\alpha)$, so we 
set $d=\dimRok(\alpha)$ and assume that $d<\infty$. 
Let $\ep>0$ and $F\subseteq B$ be a finite subset. 
Without loss of generality, we assume that $F$ contains
$1_g$ for all $g\in G$ and that it is $\beta$-invariant. 
Since $B$ is generated by the $\beta$-translations of $A$
(condition (d) in \autoref{df:enveloping}), 
there exist $a_g\in A$, for $g\in G$ such that $1_B=\sum\limits_{g\in G}\beta_g(a_g)$. Set $\ep_0= \frac{\ep}{|G|\max\limits_{g\in G}\left\|a_g\right\|}$. Using
\autoref{prop:RdimReductionUnital}, let
$f_g^{(j)}\in A_g$, for $g\in G$ and $j=0,\ldots,d$,
be positive contractions satisfying the following conditions:
\be\item[(i)] $\alpha_g(f_h^{(j)}1_{g^{-1}})=f_{gh}^{(j)}1_{g}$ for all $g,h\in G$ and $j=0,\ldots,d$;
\item[(ii)] $\|f_g^{(j)}f_h^{(j)}\|<\ep$ for all $j=0,\ldots,d$ and all $g,h\in G$ with $g\neq h$;
\item[(iii)] $\Big\|\sum\limits_{g\in G}\sum\limits_{j=0}^d f_g^{(j)}-1_A\Big\|<\ep$.
\item[(iv)] $\|f_g^{(j)}b-bf_g^{(j)}\|<\ep$ for all $g\in G$, all $j=0,\ldots,d$ and all $b\in F$.\ee

For $g\in G$ and $j=0,\ldots,d$, set
$\widetilde{f}_g^{(j)}=\beta_g(f_1^{(j)})\in B$.
We claim that the $\widetilde{f}_g^{(j)}$ are Rokhlin
towers for $\beta$ with respect to $(F, \ep)$.
Condition (1) in \autoref{df:Rdim} is clearly satisfied.
In order to check (2), let $j=0,\ldots,d$ and $g,h\in G$
with $g\neq h$ be given. Using that $f_1^{(j)}=f_1^{(j)}1_A$
at the second step, and that $1_{g^{-1}h}=1_A\beta_{g^{-1}h}(1_A)$ at the third, we get
\begin{align*}\left\|\widetilde{f}_g^{(j)}\widetilde{f}_h^{(j)}\right\|&=\left\|\beta_g(f_1^{(j)})\beta_h(f_1^{(j)})\right\|=\left\|f_1^{(j)}1_A\beta_{g^{-1}h}(1_Af_1^{(j)})\right\|\\
&=\left\|f_1^{(j)}1_{g^{-1}h}\beta_{g^{-1}h}(f_1^{(j)})\right\|
=\left\|f_1^{(j)}\alpha_{g^{-1}h}(1_{h^{-1}g}f_1^{(j)})\right\|\\
& \stackrel{\mathrm{(i)}}{=} \left\|f_1^{(j)}f_{g^{-1}h}^{(j)}\right\|\stackrel{\mathrm{(ii)}}{<} \ep_0.
\end{align*}
To check (3), 
it suffices to show that for any $a\in A$ and $h\in G$, we have 
\[\left\|\sum\limits_{j=0}^{d}\sum\limits_{g\in G}\widetilde{f}_g^{(j)}\beta_h(a)-\beta_h(a)\right\|<\|a\|\ep_0.\] 
Indeed, once this is established, and since 
$1_B=\sum\limits_{h\in G}\beta_h(a_h)$, it will 
follow that
\[\left\|\sum\limits_{j=0}^{d}\sum\limits_{g\in G}\widetilde{f}_g^{(j)}-1_B\right\|< |G|\max_{g\in G}\left\|a_g\right\|\ep_0=\ep,\]
thus establishing (3). Let $a\in A$ and $h\in G$ be given; without
loss of generality we assume that $\|a\|\leq 1$. 
Then:
\begin{align*}
\sum\limits_{j=0}^{d}\sum\limits_{g\in G}\widetilde{f}_g^{(j)}\beta_h(a)&=\sum\limits_{j=0}^{d}\sum\limits_{g\in G}\beta_g(f_1^{(j)})\beta_h(a)\\
&=\sum\limits_{j=0}^{d}\sum\limits_{g\in G}\beta_h(\beta_{h^{-1}g}(f_1^{(j)}1_A)1_Aa)\\
&=\sum\limits_{j=0}^{d}\sum\limits_{g\in G}\beta_h(\beta_{h^{-1}g}(f_1^{(j)}1_A\beta_{g^{-1}h}(1_A))a)\\
&=\sum\limits_{j=0}^{d}\sum\limits_{g\in G}\beta_h(\alpha_{h^{-1}g}(f_1^{(j)}1_{g^{-1}h})a)\\
&\stackrel{\mathrm{(i)}}{=}\sum\limits_{j=0}^{d}\sum\limits_{g\in G}\beta_h(f_{h^{-1}g}^{(j)}a)\\
&\stackrel{\!\mathrm{(iii)}}{\, \approx_\ep}_{\!_0}\! \beta_h(a),
\end{align*}
as desired. Finally, to check condition (4), let
$a\in F$, $g\in G$ and $j=0,\ldots,d$ be given. 
Using at the last step that $\beta_{g^{-1}}(a)\in F$,
we get
\begin{align*}
\left\|\widetilde{f}_g^{(j)}a-a\widetilde{f}_g^{(j)}\right\|&=\left\|\beta_g(f_1^{(j)})a-a\beta_g(f_1^{(j)})\right\|\\
&=\left\|f_1^{(j)}\beta_{g^{-1}}(a)-\beta_{g^{-1}}(a)f_1^{(j)}\right\|\stackrel{\mathrm{(iv)}}{<}\ep_0.\end{align*}
This shows that $\dimRok(\beta)\leq \dimRok(\alpha)$.
Observe that the Rokhlin towers for $\beta$ that we
constructed satisfy the following identity 
for all $j,k=0,\ldots,d$ and $g,h\in G$:
\begin{align*}
 \widetilde{f}_g^{(j)}\widetilde{f}_h^{(k)}&=\beta_g(f_1^{(j)})\beta_h(f_1^{(k)})=\beta_g\left(f_1^{(j)}1_A\beta_{g^{-1}h}(1_Af_1^{(k)})\right)\\
&=\beta_g\left(f_1^{(j)}\alpha_{g^{-1}h}(1_{h^{-1}g}f_1^{(k)})\right)
= \beta_g(f_1^{(j)}f_{g^{-1}h}^{(k)}).\end{align*}
By taking adjoints, we also get 
$\widetilde{f}_h^{(k)}\widetilde{f}_g^{(j)}=
\beta_g(f_{g^{-1}h}^{(k)}f_1^{(j)})$.
In particular, 
\[\left\| \widetilde{f}_g^{(j)}\widetilde{f}_h^{(k)}-\widetilde{f}_h^{(k)}\widetilde{f}_g^{(j)}\right\|\leq
\left\|f_1^{(j)}f_{g^{-1}h}^{(k)}-f_{g^{-1}h}^{(k)}f_1^{(j)}\right\|.
\]
Thus, if
$\cdimRok(\alpha)\leq d$ and the Rokhlin towers 
$f_g^{(j)}$ for $\alpha$ also satisfy condition
(5) in~\autoref{df:Rdim}, then the Rokhlin towers 
$\widetilde{f}_g^{(j)}$ for $\beta$ also satisfy (5) 
and hence 
$\cdimRok(\beta)\leq d$.
\end{proof}

\section{Structure of the crossed product}\label{sec: StructureCrossProd}
Since its introduction in \cite{HirWinZac_rokhlin_2015}, Rokhlin
dimension has predominantely been used to study structural properties
of the associated crossed products. 
Of greatest relevance are those properties
related to the classification programme for nuclear \ca s, such as
the UCT,
finiteness of the nuclear dimension, absorption of a strongly
self-absorbing \ca \cite{TomWin_strongly_2007}, or divisibility properties on $K$-theory. 

In this section, we explore the structure of the 
crossed products and fixed point algebras 
of partial actions with finite Rokhlin dimension.
Our approach makes use of the decomposition property introduced and 
studied in \cite{AbaGarGef_decomposable_2020}, which we recall
in Subsection~4.1. For unital partial actions, a more
direct argument can be given, which even yields better results
(notably in the zero-dimensional case).
In this setting, we show that the crossed product is Morita equivalent
to the fixed point algebra, a fact that fails for general
partial actions of finite Rokhlin dimension.

\subsection{The decomposition property}
An important ingredient in our study of partial actions with
finite Rokhlin dimension is our previous 
work~\cite{AbaGarGef_decomposable_2020} on decomposable partial 
actions. For the convenience of the reader, 
we make here a small digression. 

\begin{df}\label{df:TnG}
Let $G$ be a finite group.
Given $n=1,\ldots,|G|$, we define the \emph{space of $n$-tuples} 
of $G$ to be 
\[\mathcal{T}_n(G)=\{\tau \subseteq G \colon 1\in \tau \mbox{ and } |\tau|=n\}.\]
For $g\in G$, we set 
$\mathcal{T}_n(G)_{g}=\{\tau\in \mathcal{T}_n(G)\colon g\in \tau\}$.
There is a canonical partial action 
$\texttt{Lt}$ of $G$ on $\mathcal{T}_n(G)$, with 
$\texttt{Lt}_g\colon \mathcal{T}_n(G)_{g^{-1}}\to \mathcal{T}_n(G)_g$ induced by left translation by $g$.
For $\tau \in \mathcal{T}_n(G)$, we write $G\cdot \tau\subseteq \mathcal{T}_n(G)$ for the orbit of
$\tau$ with respect to $\texttt{Lt}$.
\end{df}

We will adopt the following convention.
Let $\alpha$ be a 
partial action of a finite group $G$ on a \ca\ $A$, 
and let $n=1,\ldots,|G|$.
For $\tau\in \mathcal{T}_n(G)$, we write $A_{\tau}$ for the ideal $A_{\tau}=\bigcap_{g\in \tau} A_{g}$.
Then
$\alpha_{g}(A_{\tau})=A_{g\tau}$ for
$g\in G$ and $\tau\in \mathcal{T}_n(G)_{g^{-1}}$
For $\tau \in \mathcal{T}_n(G)$, we 
set $A_{G\cdot\tau}=\sum_{g\in \tau^{-1}} A_{g\tau}$.
When $A=C_0(X)$ for a locally
compact Hausdorff space $X$, we write $X_\tau$ for the 
spectrum of $C_0(X)_\tau$, and identify it canonically
with $\bigcap_{g\in \tau}X_g$. We use similar notation
for $X_{G\cdot \tau}$.

\begin{df}\label{df:nIntProp}
Let $G$ be a finite group, let $A$ be a \ca, and let 
$\alpha=((A_g)_{g\in G}, (\alpha_g)_{g\in G})$ be a partial action of $G$ on $A$. 
Given $n=1,\ldots,|G|$, 
we say that $\alpha$ has the 
\emph{$n$-decomposition property} if
\be\item[(a)] $A=\overline{\sum\limits_{\tau\in\mathcal{T}_n(G)}A_\tau}$, and 
\item[(b)]$A_{\tau}\cap A_g= \{0\}$ for all $\tau\in\mathcal{T}_{n}(G)$ and all $g\in G$ such that $g\notin \tau$. \ee
We say that $\alpha$ is \emph{decomposable} 
if it has the $n$-decomposition property for some $n\in \N$. 
A partial action on a locally compact space $X$ is said to have
the \emph{$n$-decomposition property} if the induced partial action on $C_0(X)$ has it.  
\end{df} 

\begin{nota}\label{nota:Section}
Adopt the notation from \autoref{df:nIntProp}. 
For $\tau\in\mathcal{T}_n(G)$,
we set 
$H_\tau=\{h\in G\colon h\tau=\tau\}$ and
$m_\tau=\frac{n}{|H_\tau|}-1$.
Using Lemma~2.8 in~\cite{AbaGarGef_decomposable_2020}, we 
fix elements $x_0^\tau=1, x_1^\tau,\ldots,x_{m_\tau}^\tau\in G$ satisfying
\[\tau=H_\tau\sqcup H_\tau x^\tau_1\sqcup \ldots\sqcup H_\tau x^\tau_{m_\tau}.\]
Whenever $\tau$ is understood from the context, we will omit it from the notation
for $H_\tau$, $m_\tau$ and $x_j^\tau$, for $j=1,\ldots,m_\tau$. 
Let $\mathcal{O}_n(G)$ be the orbit space for the partial system described in \autoref{df:TnG}. We denote by $\kappa\colon \mathcal{T}_n(G)\to \mathcal{O}_n(G)$ the canonical
quotient map, and fix, for the rest of this work, a global section
$\eta\colon \mathcal{O}_n(G)\to \mathcal{T}_n(G)$ for it. 
For $z\in\mathcal{O}_n(G)$, we write $\tau_z$ for $\eta(z)$; we write 
$H_z$ for $H_{\tau_z}$; 
and $m_z$ for $m_{\tau_z}$. 
\end{nota}

The following is part of Proposition~2.11 in~\cite{AbaGarGef_decomposable_2020}.

\begin{prop}\label{prop:EquivDecomp}
Let $G$ be a finite group, let $A$ be a \ca, let $n=1,\ldots,|G|$, let 
$\alpha$ be a partial action of $G$ on $A$ with the
$n$-decomposition property, and let 
$\tau\in\mathcal{T}_n(G)$. Adopt the conventions from \autoref{nota:Section}. Then:
\be
\item The restriction of $\alpha|_{H_\tau}$ to $A_\tau$ is a global action;
\item There is a natural $G$-equivariant isomorphism
\[\varphi\colon \bigoplus_{z\in\mathcal{O}_n(G)} A_{G\cdot \tau_z} \to A\]
given by $\varphi(a)=\sum_{z\in \mathcal{O}_n(G)} a_z$ for all
$a=(a_z)_{z\in\mathcal{O}_n(G)}$.
\ee
\end{prop}

By part~(2) above, many facts
about decomposable partial actions can be reduced to
the $G$-invariant direct summands $A_{G\cdot\tau}$.
In particular, for many purposes it suffices
to work with a single $\tau\in\mathcal{T}_n(G)$ 
and the induced partial action on $A_{G\cdot \tau}$.

Next, we recall Theorem~6.1 from~\cite{AbaGarGef_decomposable_2020},
which asserts that every partial action of a
finite group is canonically an iterated extension of 
decomposable partial actions. It follows that many aspects
about partial actions of finite groups can be reduced to
the case of decomposable partial actions, as long as one
has control over the resulting equivariant extension 
problem (which is in general quite complicated). 

\begin{thm}\label{thm: FinGrpIntersProp}
Let $G$ be a finite group, let $A$ be a \ca, and let 
$\alpha$ be a partial action of $G$ on $A$. 
Then there are canonical equivariant extensions 
\[\xymatrix{0\ar[r] & (D^{(k)},\delta^{(k)}) \ar[r]& (A^{(k)},\alpha^{(k)})\ar[r] & (A^{(k-1)},\alpha^{(k-1)})\ar[r]&0,}
\]
for $2\leq k \leq |G|$, satisfying the following properties 
\be
\item[(a)] $A^{(|G|)}=A$ and $\alpha^{|G|}=\alpha$;
\item[(b)] $\delta^{(k)}$ has the $k$-decomposition property;
\item[(c)] $\alpha^{(1)}$ has the 1-decomposition property.
\ee
\end{thm}

We close this subsection by showing that the Rokhlin dimension of 
a decomposable partial action can be computed in terms of 
the global subsystems
$H_\tau \curvearrowright A_{\tau}$.

\begin{thm}\label{thm:dimRokResH}
Let $G$ be a finite group, let $A$ be a \ca, let $n=1,\ldots,|G|$,
and let $\alpha$ be a partial action of a finite group $G$ on $A$
with the $n$-decomposition property.
Fix $\tau\in\mathcal{T}_n(G)$. Then
\[\dimRok(\alpha|_{H_\tau})= \dimRok(\alpha|_{A_{G\cdot \tau}}) \ \ \mbox{ and } \ \ \cdimRok(\alpha|_{H_\tau})=\cdimRok(\alpha|_{A_{G\cdot \tau}}).\]
Consequently,
\[\dimRok(\alpha)= \max_{\tau\in\mathcal{T}_n(G)}{\dimRok(\alpha|_{H_\tau})} \ \ \mbox{ and } \ \ \cdimRok(\alpha)=\max_{\tau\in\mathcal{T}_n(G)}{\cdimRok(\alpha|_{H_\tau})}.\]
\end{thm}
\begin{proof}
We begin by observing that the last two identities are consequences
of the first two, by part~(2) of \autoref{prop:EquivDecomp}.

We give a proof for the first equality; the proof for $\cdimRok$ is similar. We start by showing $\dimRok(\alpha|_{A_{G\cdot\tau}})\leq \dimRok(\alpha|_{H_\tau})$. For this
we set $d=\dimRok(\alpha|_{H_\tau})$ and assume that $d<\infty$. 
Adopt \autoref{nota:Section}, and fix
$x_0=1,x_1,\ldots, x_m\in G$ with 
\[\tau= H_\tau\sqcup H_\tau x_1\sqcup\ldots\sqcup H_\tau x_m.\]
Let $F\subseteq A_{G\cdot\tau}$ be a finite subset consisting
of contractions, and let $\ep>0$. 
Since $A_{G\cdot \tau}$ can be canonically identified with 
$\bigoplus_{\ell=0}^{m} A_{x_\ell^{-1}\tau}$, we may assume that $F$ can be written as a disjoint
union $F=F_0\sqcup \ldots\sqcup F_m$, where $F_\ell\subseteq A_{x_\ell^{-1}\tau}$ for every $\ell=0,\ldots,m$.
Set $K=\bigcup_{\ell=1}^m \alpha_{x_\ell}(F_\ell)\subseteq A_{\tau}$, and let $\xi_h^{(j)}\in A_\tau$,
for $h\in H_\tau$ and $j=0,\ldots,d$, 
be Rokhlin towers for $\alpha|_{H_\tau}$ with respect to 
$(K,\ep)$. (Note that $\xi_h^{(j)}\neq 0$ for all $h\in H_\tau$ 
and all $j=0,\ldots,d$.)
For $g\in G$ and $j=0,\ldots,d$, we set 
\[f_g^{(j)}=\begin{cases}
\alpha_{x_\ell^{-1}}(\xi_h^{(j)}), & \text{ if } g=x_\ell^{-1}h, \text{ for some } \ell=0,\ldots, m, \text{ and } h\in H_\tau\\
0, & \text{otherwise}.
\end{cases} \]
Observe that $f_g^{(j)}$ is well-defined (because 
$x_\ell^{-1}H_\tau\cap x_{r}^{-1}H_\tau=\emptyset$ if $\ell\neq r$),
and that it is a positive contraction in 
$A_{G\cdot\tau}\cap A_g$.
We claim that the $f_g^{(j)}$ satisfy the conditions in \autoref{df:Rdim}, and thus witness the fact that 
$\dimRok(\alpha|_{A_{G\cdot\tau}})\leq \dimRok(\alpha|_{H_\tau})$. 

In order to check (1), Let $g,k\in G$, $j=0,\ldots,d$, let 
$\ell=0,\ldots,m$, and $ x\in A_{g^{-1}}\cap F_\ell$. We need to show 
that 
\[\label{eqn:4.1}\tag{4.1}\left\|\alpha_g(f_{k}^{(j)}x)-f_{gk}^{(j)}\alpha_g(x) \right\|<\ep.\]
Since by the decomposition property (specifically condition (b)
in \autoref{df:nIntProp}), 
the element $x\in A_{g^{-1}}\cap F_\ell$ is zero (and hence the 
inequality holds trivially) whenever
$g^{-1}\notin x_{\ell}^{-1}\tau$, it suffices to assume that
there exist unique $h\in H_\tau$ and $r=0,\ldots,m$ such that 
$g=x_r^{-1}hx_\ell$. By construction, we have
$f_k^{(j)}=0$ unless $k\in \tau^{-1}$, and similarly
$f_{gk}^{(j)}=0$ unless $gk\in\tau^{-1}$. We accordingly
divide the proof into three cases:

\textbf{Case~1:} $k\notin \tau^{-1}$, so that
$f_k^{(j)}=0$. We claim that $f_{gk}^{(j)}\alpha_g(x)=0$.
Arguing by contradiction, assume that the product 
$f_{gk}^{(j)}\alpha_g(x)$, which 
belongs to $A_{gk}\cap A_\tau$, is 
nonzero. By the intersection property, we must have $gk\in \tau$
and thus there are unique $s=0,\ldots,m$ and $h_1\in H_\tau$ with $gk=x_s^{-1}h_1$. 
Since  
$f_{gk}^{(j)}\alpha_g(x)$, which belongs to $A_{x_r^{-1}\tau}\cap A_{x_{s}^{-1}\tau}$, is nonzero, the decomposition property implies that $r=s$. Thus 
\[k=g^{-1}x_s^{-1}h_1=x_t^{-1}hx_rx_s^{-1}h_1=x_\ell^{-1}hh_1\in x_\ell^{-1}H\subseteq \tau^{-1},\] 
thus contradicting our assumption. This 
verifies $(\ref{eqn:4.1})$ in this case.

\textbf{Case~2:} $gk\notin \tau^{-1}$, so that $f_{gk}^{(j)}=0$. We claim that $\alpha_g(f_k^{(j)}x)=0$. 
Arguing by contradiction, assume that the product 
$f_k^{(j)}x$, which 
belongs to $A_{k}\cap A_{x_\ell^{-1}\tau}$, is 
nonzero. By the intersection property, we must have $k\in x_\ell^{-1}\tau$,
and thus there is $h_2\in H_\tau$ with $k=x_\ell^{-1}h_2$. Thus 
\[gk= x_r^{-1}hx_\ell x_\ell^{-1}h_2=x_r^{-1}hh_2\in x_r^{-1}H\subseteq \tau^{-1},
\]
thus contradicting our assumption. This 
verifies $(\ref{eqn:4.1})$ in this case.

\textbf{Case~3:} $gk,k \in \tau^{-1}$. 
Then there exists $\widetilde{h}\in H_\tau$ with $k=x_\ell^{-1}\widetilde{h}$, so that $gk=x_{r}^{-1}h\widetilde{h}$. 
Thus, 
\begin{align*}
\left\| \alpha_g(f_k^{(j)}x)-f_{gk}^{(j)}\alpha_g(x)\right\|
=\Big\|\underbrace{\alpha_{x_r^{-1}hx_\ell}(\alpha_{x_\ell^{-1}}(\xi_{\widetilde{h}}^{(j)})x)}_{\approx_{\ep} \
\alpha_{x_r^{-1}}(\xi_{h\widetilde{h}}^{(j)})\alpha_{g}(x)}-\alpha_{x_r^{-1}}(\xi_{h\widetilde{h}}^{(j)})\alpha_{g}(x)\Big\|
< \ep,
\end{align*}
thus verifying $(\ref{eqn:4.1})$ also in this case. 

We turn to condition (2) in~\autoref{df:Rdim}
Let $g,k\in G$ with $g\neq k$, let $j=0,\ldots,d$, 
let $\ell=0,\ldots,m$ and let $a\in F_\ell$. We need to show that 
$f_{g}^{(j)}f_k^{(j)}a$ has norm at most $\ep$. 
We may assume that there exist $r,s=0,\ldots,m$ and 
$h_1,h_2\in H_\tau$ with $g=x_r^{-1}h_1$ and $k=x_s^{-1}h_2$ 
(or else either $f_g^{(j)}=0$ or $f_k^{(j)}=0$).
Additionally, since the product $f_{g}^{(j)}f_k^{(j)}a$
belongs to $A_{g}\cap A_k\cap A_{x_\ell^{-1}\tau}$,
we may assume that $r=s=\ell$ (or else $f_{g}^{(j)}f_k^{(j)}a=0$). In this case, we have 
\begin{align*}
 f_{g}^{(j)}f_k^{(j)}a= \alpha_{x_\ell^{-1}}(\xi_{h_1}^{(j)})\alpha_{x_\ell^{-1}}(\xi_{h_2}^{(j)})a
=\alpha_{x_\ell^{-1}}\Big(\underbrace{\xi_{h_1}^{(j)}\xi_{h_2}^{(j)}\alpha_{x_\ell}(a)}_{\approx_{\ep} \ 0}\Big)\approx_\ep 0,
\end{align*}
as desired. In order to check condition (3), 
let $\ell=0,\ldots,m$ and $a\in F_\ell$. Using at the second
step that a product of the form $\alpha_{x_r^{-1}}(\xi_h^{(j)})a$ is zero unless $r=\ell$, we have
\begin{align*}
\sum\limits_{j=0}^{d}\sum\limits_{g\in G} f_g^{(j)}a&=\sum\limits_{j=0}^{d}\sum\limits_{h\in H_\tau}\sum\limits_{r=0}^{m} \alpha_{x_r^{-1}}(\xi_h^{(j)})a\\
&=\sum\limits_{j=0}^{d}\sum\limits_{h\in H_\tau}\alpha_{x_\ell^{-1}}(\xi_h^{(j)})a\\
&=\sum\limits_{j=0}^{d}\sum\limits_{h\in H_\tau}\alpha_{x_\ell^{-1}}(\xi_h^{(j)}\alpha_{x_\ell}(a))\\
&=\alpha_{x_{\ell}^{-1}}\left(\sum\limits_{j=0}^{d}\sum\limits_{h\in H_\tau}\xi_h^{(j)}\alpha_{x_\ell}(a)\right)\approx_{\ep} \alpha_{x_\ell}(a),
\end{align*}
as desired. Finally, to check condition (4), let $g\in G$,
$j=0,\ldots,d$, $\ell=0,\ldots,m$ and $a\in F_\ell$ be given.
Since $f_g^{(j)}=0$ unless $g\in \tau^{-1}$, we may assume
that there are $h\in H_\tau$ and $s=0,\ldots,m$ such that 
$g=x_s^{-1}h$.
Since the products $f_g^{(j)}a$ and $af_g^{(j)}$ are both zero
unless $g\in x_{\ell}^{-1}\tau$, we may assume that $s=\ell$.
For $b\in F$, we have
\begin{align*}
\left\|\big(f_g^{(j)}a-af_g^{(j)}\big)b\right\|&=\left\|\big(\alpha_{x_\ell^{-1}}(\xi_h^{(j)})a-a\alpha_{x_\ell^{-1}}(\xi_h^{(j)})\big)b\right\|\\
&=\left\|\big(\xi_h^{(j)}\alpha_{x_\ell}(a)-\alpha_{x_\ell}(a)\xi_h^{(j)}\big)\alpha_{x_\ell}(b)\right\|<\ep.
\end{align*}
This completes the proof that $\dimRok(\alpha|_{A_{G\cdot\tau}})\leq \dimRok(\alpha|_{H_\tau})$. 

Next, we show $\dimRok(\alpha|_{H_\tau})\leq \dimRok(\alpha|_{A_{G\cdot\tau}})$. Set 
$d=\dimRok(\alpha|_{A_{G\cdot\tau}})$ and 
assume that $d<\infty$. 
Fix an approximate unit $(e_{\lambda})_{\lambda\in\Lambda}$ of $A_\tau$, and let $\pi\colon A_{G\cdot\tau}\to A_{\tau}$ be the quotient map given by $\pi(x)=\lim_{\lambda}xe_\lambda$
for all $x\in A_{G\cdot\tau}$. (The map $\pi$ does not depend
on the approximate identity, but this is not relevant to us.) Note that
\[\label{eqn:4.2}\tag{4.2}
\pi(x)a=xa 
\]
for all $x\in A_{G\cdot\tau}$ and all $a\in A_\tau$. 
Let $F\subseteq A_\tau$ be a finite subset and let $\ep>0$. We may assume without loss of generality that $F$ is 
$H_\tau$-invariant and consists of contractions.
For every right coset $q\in H_\tau\setminus G$, let $g_q\in G$
satisfy $H_\tau g_q=q$. Set $\ep_0=\frac{\ep}{[G:H_\tau]^2}$ and 
let $\xi_g^{(j)}\in A_{G\cdot\tau}\cap A_g$, for 
$g\in G$ and $j=0,\ldots, d$, be Rokhlin towers for 
$\alpha|_{A_{G\cdot \tau}}$ with respect to $(F,\ep_0)$.

For $h\in H_\tau$ and $j=0,\ldots, d$, we set
\[f_h^{(j)}= \sum\limits_{q\in H_\tau\setminus G}\pi(\xi^{(j)}_{hg_q})\in A_\tau. \]
We claim that these positive contractions witness 
that $\dimRok(\alpha|_{H_\tau})\leq d$ for $(F,\ep)$.

Let $h_1,h_2\in H_\tau$, let $j=0,\ldots, d$ and let $a\in F$ 
be given. Then
\begin{align*}
\left\| \alpha_{h_1}(f_{h_2}^{(j)})a-f_{h_1h_2}a\right\|&\leq \sum\limits_{q\in H_\tau\setminus G}\left\|
\alpha_{h_1}(\pi(\xi_{h_2g_q}^{(j)}))a-\xi_{h_1h_2g_q}^{(j)}a\right\|\\
&\stackrel{(\ref{eqn:4.2})}{=}\sum\limits_{q\in H_\tau\setminus G}\left\|
\alpha_{h_1}(\xi_{h_2g_q}^{(j)}\alpha_{h_1^{-1}}(a))-\xi_{h_1h_2g_q}^{(j)}a\right\|\\
&\leq [G:H_\tau]\ep_0\leq \ep,
\end{align*}
thus establishing condition (1) in~\autoref{df:Rdim}. In order
to prove (2), 
let $h_1,h_2\in H_\tau$ with $h_1\neq h_2$, let $j=0,\ldots ,d$, and let $a\in F$. Then
\begin{align*}
\left\| f_{h_1}^{(j)}f_{h_2}^{(j)}a\right\|=\Big\| \sum\limits_{p,q\in H_\tau\setminus G}\xi_{h_1g_q}^{(j)} \xi_{h_2g_p}^{(j)}a\Big\|
\leq [G:H_\tau]^2 \ep_0=\ep,
\end{align*}
using at the second step that $h_1g_q\neq h_2q_p$ for all $p,q\in H_\tau\setminus G$. To check (3), let $a\in F$. Then
\begin{align*}
\sum\limits_{j=0}^{d}\sum\limits_{h\in H_\tau}f_h^{(j)}a&= 
\sum\limits_{j=0}^{d}\sum\limits_{h\in H_\tau}\sum\limits_{q\in H_\tau\setminus G}\pi(\xi_{hg_q}^{(j)})a\\
&\stackrel{(\ref{eqn:4.2})}{=}\sum\limits_{j=0}^{d}\sum\limits_{h\in H_\tau}\sum\limits_{q\in H_\tau\setminus G}\xi_{hg_q}^{(j)}a\\
&=\sum\limits_{j=0}^{d}\sum\limits_{g\in G}\xi_{g}^{(j)}a\approx_\ep a.
\end{align*}
Finally, to check (4), let $h\in H_\tau$, $j=0,\ldots, d$, 
and $a,b\in F$ be given. Then
\begin{align*}
\left\| \big(f_h^{(j)}a-af_h^{(j)}\big)b\right\|&=\left\| \sum\limits_{q\in H_\tau\setminus G}\big(\xi_{hg_q}^{(j)}a-a\xi_{hg_q}^{(j)}\big)b\right\|\leq [G:H_\tau]\ep_0\leq \ep,
\end{align*}
as required. This shows that $\dimRok(\alpha|_{H_\tau})\leq d$, and completes the proof.
\end{proof}

\subsection{Crossed products and fixed point algebras}
Our next result shows
that a number of properties are preserved by formation of 
crossed products and fixed point algebras by partial actions 
with finite Rokhlin dimension. 
The kind of properties preserved in this setting is more 
restrictive than in the global setting, particularly 
since the property in question must pass to extensions. For 
\emph{unital} partial actions, we will show in 
\autoref{thm:CPRdimUnital} that even more
properties are preserved. 

Recall that if 
$\alpha$ is a partial action of a finite group $G$ on a \ca\ $A$, then 
its \emph{crossed product} $A\rtimes_\alpha G$ is the 
set of all formal linear combinations of elements of the form $a_gu_g$,
where $g\in G$ and $a_g\in A_g$, subject to the relations
\[a_gu_gb_hu_h=\alpha_g(\alpha_{g^{-1}}(a_g)b_h)u_{gh} \ \ \mbox{ and } \ \
 (a_gu_g)^*=\alpha_{g^{-1}}(a_g^*)u_{g^{-1}}.\]
We consider $A\rtimes_\alpha G$ with its greatest $C^*$-norm, which is
not hard to see to exist.   
\par Moreover, its \emph{fixed point algebra} $A^\alpha$ is defined as
\[A^\alpha=\{x\in A\colon  \alpha_g(xa_{g^{-1}})=x\alpha_{g}(a_{g^{-1}}) \mbox{ for all } g\in G \mbox{ and all } a_{g^{-1}}\in A_{g^{-1}}\}.\]



\begin{thm}\label{thm:CPRdim} 
Let $G$ be a finite group, and let $d\in \N$. 
Let \textbf{P} be a property for \ca s which is preserved by:
\bi\item[(E)] passage to ideals, quotients and extensions; 
\item[(M)] Morita equivalence;
\item[(C)] crossed products by \emph{global} actions
of $G$ with $\dimRok\leq d$. 
\ei
Let $A$ be a unital \ca, and let 
$\alpha$ be a partial action of $G$ on $A$. If 
$\dimRok(\alpha)\leq d$ and $A$ satisfies
\textbf{P}, then so do $A\rtimes_\alpha G$ and $A^\alpha$. 
In particular,
\be
\item If $\dimRok(\alpha)<\I$ and $\dimnuc(A)<\I$, then
$\dimnuc(A^\alpha),\dimnuc(A\rtimes_\alpha G)<\I$. Indeed,
\[\dimnuc(A\rtimes_\alpha G)\leq (|G|-1)(\dimRok(\alpha)+1)(\dimnuc(A)+1)+\dimnuc(A).\]
\item If $\dimRok(\alpha)<\I$ and $\dr(A)<\I$, then
$\dr(A^\alpha), \dr(A\rtimes_\alpha G)<\I$. Indeed,
\[\dr(A\rtimes_\alpha G)\leq (|G|-1)(\dimRok(\alpha)+1)(\dr(A)+1)+\dr(A).\]
\ee

A similar statement
is true for $\cdimRok$. In
particular, 
\be
\item[(3)] If $\cdimRok(\alpha)<\I$ and $\mathrm{sr}(A)<\I$, then 
$\mathrm{sr}(A^\alpha),\mathrm{sr}(A\rtimes_\alpha G)<\I$. Indeed,
\[\mathrm{sr}(A\rtimes_\alpha G)\leq \frac{|G|(\sr(A)+\cdimRok(\alpha)+3)-2}{2}.
\]
\item[(4)]
If $\cdimRok(\alpha)<\I$ and $\mathrm{RR}(A)<\I$, then 
$\mathrm{RR}(A^\alpha),\mathrm{RR}(A\rtimes_\alpha G)<\I$. 
\item[(5)] Let $\mathcal{D}$ be a strongly self-absorbing \ca. If $\cdimRok(\alpha)<\I$ and $A$ is $\mathcal{D}$-absorbing, 
then $A^\alpha$ and $A\rtimes_\alpha G$ are $\mathcal{D}$-absorbing as well. 
\ee
\end{thm}
\begin{proof}
Let \textbf{P} be a property as in the statement, let $A$ be a \ca\ satisfying
\textbf{P}, and let $\alpha$ be 
a partial action of $G$ on $A$ with $\dimRok(\alpha)\leq d$. 
By \autoref{thm: FinGrpIntersProp}, there
there are canonical equivariant extensions 
\[\label{eqn:4.3}\tag{4.3}
\xymatrix{0\ar[r] & (D^{(k)},\delta^{(k)}) \ar[r]& (A^{(k)},\alpha^{(k)})\ar[r] & (A^{(k-1)},\alpha^{(k-1)})\ar[r]&0,}
\]
for $2\leq k \leq |G|$, satisfying the following properties 
\be 
\item[(D.1)] $A^{(|G|)}=A$ and $\alpha^{|G|}=\alpha$;
\item[(D.2)] $\delta^{(k)}$ has the $k$-decomposition property;
\item[(D.3)] $\alpha^{(1)}$ has the 1-decomposition property.
\ee

In particular, each $A^{(k)}$ is a quotient of $A$,
and each $D^{(k)}$ is an ideal of a quotient of $A$. By (E), 
all of these \ca s satisfy \textbf{P}. 
By repeatedly applying \autoref{prop:PermProp}, we deduce that
\[\label{eqn:4.4}\tag{4.4}\dimRok(\alpha^{(k)})\leq d \ \ \text{ and } \ \ \dimRok(\delta^{(k)})\leq d\]
for all $k=2,\ldots,|G|$, while $\dimRok(\alpha^{(1)})=0$ by
(D.3) and \autoref{eg:trivialRp}.
For $k=2,\ldots,|G|$, apply crossed products
to ($\ref{eqn:4.3}$), to get the extension
\[\label{eqn:4.5}\tag{4.5}
\xymatrix{0\ar[r] & D^{(k)}\rtimes_{\delta^{(k)}}G \ar[r]& A^{(k)}\rtimes_{\alpha^{(k)}}G \ar[r] & A^{(k-1)}\rtimes_{\alpha^{(k-1)}}G\ar[r]&0.}
\]

\textbf{Claim:} \emph{$A^{(k)}\rtimes_{\alpha^{(k)}} G$ satisfies \textbf{P} for all $k=1,\ldots,|G|$.}
We prove this by induction on $k$.
Since $A^{(1)}\rtimes_{\alpha^{(1)}} G=A^{(1)}$ by
\autoref{eg:trivialRp} and $A^{(1)}$ is a quotient of $A$, 
this follows from (E). Assume we have proved it for $k-1$, and let us prove
it for $k$. 
Since \textbf{P} passes to extensions by (E), the exact sequence
in ($\ref{eqn:4.5}$) 
implies that it suffices to show that
$D^{(k)}\rtimes_{\delta^{(k)}} G$ satisfies \textbf{P}.
Combining (D.2) and Theorem~C 
in~\cite{AbaGarGef_decomposable_2020}, it follows that 
$D^{(k)}\rtimes_{\delta^{(k)}} G$ is isomorphic to a finite
direct sum of algebras of the form 
$M_{m_\tau}(D_\tau^{(k)}\rtimes_{\delta_\tau^{(k)}} H_\tau)$, for 
$m_\tau=k/|H_\tau|\leq |G|$ 
and $\tau\in\mathcal{T}_k(G)$, where $D^{(k)}_\tau$
is an ideal in $D^{(k)}$, $H_\tau$ is a subgroup of $G$, and 
$\delta_\tau^{(k)}$ is the \emph{global} action obtained as the 
restriction of $\delta^{(k)}$ to $H_\tau$ 
and to $D^{(k)}_\tau$. Thus, by (M) it suffices to show 
$D_\tau^{(k)}\rtimes_{\delta_\tau^{(k)}} H_\tau$ 
satisfies \textbf{P} for every $\tau\in\mathcal{T}_k(G)$. 
Using \autoref{thm:dimRokResH} at the first step, we have
\[\dimRok(\delta^{(k)}_\tau)\leq \dimRok(\delta^{(k)})
\stackrel{(\ref{eqn:4.4})}{\leq} d\]
Moreover, $D^{(k)}_\tau$ satisfies \textbf{P} by (E), 
since it is an ideal in $D^{(k)}$. It follows from (C)
that $D_\tau^{(k)}\rtimes_{\delta_\tau^{(k)}} H_\tau$ satisfies
\textbf{P}, as desired. This proves the claim.

Since $A\rtimes_\alpha G$ equals 
$A^{(|G|)}\rtimes_{\alpha^{(|G|)}} G$ by (D.1), 
this proves the first assertion
in the theorem. Note that an identical argument applies to
$\cdimRok$ in place of $\dimRok$. Moreover, the argument for 
fixed point algebras is analogous, by applying fixed point algebras
to the extensions in (\ref{eqn:4.4}), using Theorem~4.5
in~\cite{AbaGarGef_decomposable_2020} instead of Theorem~C there,
and using the fact that fixed point algebra and crossed product
are Morita equivalence in the global case.
We omit the details.

The properties listed in (1) through (5) are well known to
satisfy (E) and (M); see respectively \cite{WinZac_nuclear_2010, KirWin_covering_2004, Rie_dimension_1983, BroPed_algebras_1991}.
Finally, they also satisfy (C) by Corollary~4.25 in~\cite{GarLup_applications_2018} and Theorem~3.20 in~\cite{GarHirSan_Rokhlin_2017}. 

For the estimates in 
(1), (2) and (3), one combines the estimates from 
Corollary~4.25 in~\cite{GarLup_applications_2018} 
and Theorem~3.20 in~\cite{GarHirSan_Rokhlin_2017}, with the 
known estimates for $\dimnuc$, $\dr$, or $\mathrm{sr}$
of an extension or of $A\otimes M_n$, and applies these a total of
$|G|-1$ times to the extensions in ($\ref{eqn:4.5}$). 
We omit the details.
\end{proof}

We note in passing that if $\alpha$ is a partial action of a finite
group $G$ on $A$, and $A$ satisfies \textbf{P}, then $A\rtimes_\alpha
G\rtimes 
_{\hat{\alpha}}G$ also satisfies \textbf{P}. Indeed, $A$ is Morita
equivalent to an ideal $I$ of $A\rtimes_\alpha
G\rtimes_{\hat{\alpha}}G$ whose linear orbit under 
the bidual action $\Hat{\hat{\alpha}}$ is all of $A\rtimes_\alpha G\rtimes
_{\hat{\alpha}}G$ (\cite[Theorem~6.1]{Aba_enveloping_2003}). Since
\textbf{P} is preserved by Morita equivalence and by passage to
ideals, quotients and extensions, then 
$I$ and its translates $\Hat{\hat{\alpha}}_g(I)$, and each sum
$\Hat{\hat{\alpha}}_g(I)+\Hat{\hat{\alpha}}_{g'}(I)$, also satisfy     
\textbf{P}. So $A\rtimes_\alpha G\rtimes
_{\hat{\alpha}}G=\sum_{g\in G}\Hat{\hat{\alpha}}_g(I)$ satisfies
\textbf{P}.           

The fact that the preservation results for $A^\alpha$ and 
$A\rtimes_\alpha G$ from \autoref{thm:CPRdim} look identical suggests
that there may be a tighter connection between these algebras.
Indeed, in the global setting we have $A^\alpha\sim_M A\rtimes_\alpha G$
whenever $G$ is abelian and $\dimRok(\alpha)<\I$; see 
Corollary~1.18 in~\cite{GarHirSan_Rokhlin_2017}. 
The situation for partial actions is much more complicated, 
and there exist partial actions of finite groups with finite Rokhlin
dimension such that $A^\alpha$ is not Morita equivalent to 
$A\rtimes_\alpha G$:




\begin{eg}\label{eg:NotMEq}
Set $X=(0,2]$ and $U=(0,1)\cup (1,2)\subseteq X$, 
and let $\sigma\in \mathrm{Homeo}(U)$ be given by $\sigma(x)=x+1 
\, \mathrm{ mod }\, 2$ for all $x\in U$. 
Let $\alpha$ be the partial action of $G=\Z_2=\{-1,1\}$ on 
$A=C_0(X)$ induced by $\sigma$.
This action is considered in Example~5.2 of~\cite{AbaGarGef_decomposable_2020}, where 
it is shown that  $C_0(X)^{\Z_2}$ is not Morita equivalent to 
$C_0(X)\rtimes_\alpha\Z_2$. 

We claim that $\dimRok(\alpha)<\I$. Let $F\subseteq C_0(X)$
be a finite set, and let $\ep>0$. 
For $\delta>0$, let $f_\delta\in C_0(X)$ and $e_\delta\in C_0(U)$
be given by
\[f_\delta(x)=\begin{cases*}
      1 & if $x\geq \delta$ \\
      \mathrm{linear} & otherwise
    \end{cases*} \, \mbox{and } \ 
    e_\delta(x)=\begin{cases*}
      1 & if $x\in [\delta,1-\delta]\cup [1+\delta,2-\delta]$ \\
      \mathrm{linear} & otherwise.
    \end{cases*}\]
Find $\delta\in (0,1/4)$ such that $\|e_\delta a-a\|<\ep$ for all
$a\in F\cap C_0(U)$ and $\|f_\delta a-a\|<\ep$ for all $a\in F$. 
Set 
\[f_{-1}^{(0)}=e_\delta|_{(0,1)}, \ f_1^{(0)}=e_\delta|_{(1,2]}, \
f_{-1}=(f_\delta-e_\delta)_{(0,1)}, \ \mbox{ and } \ f_1^{(1)}=(f_\delta-e_\delta)_{(1,2]}\]
It is easy to check that the above positive contractions are Rokhlin
towers for $(F,\ep)$, thus showing that $\dimRok(\alpha)\leq 1$,
as desired. 
(In fact, one can also show that $\dimRok(\alpha)\neq 0$,
so that $\dimRok(\alpha)$ is actually 1.)
\end{eg}

\subsection{Unital partial actions}
For unital partial actions, 
a different approach can be used to obtain 
information about the crossed product, leading to 
a result which is stronger than \autoref{thm:CPRdim}, 
in that here we only demand that the property in question pass to
direct sums and summands, as opposed to general extensions and 
ideals; see \autoref{thm:CPRdimUnital}. 
Most significantly, this implies that the UCT is preserved
in this setting; note that the UCT is not in general known to pass 
to ideals or quotients, and hence does not satisfy condition (E)
in \autoref{thm:CPRdim}. The same applies to several
properties of the $K$-theory.
Additionally, we isolate the case of 
Rokhlin dimension zero, for which considerably 
more can be said with this approach.

In contrast with \autoref{eg:NotMEq},
we show next that $A^\alpha$ and $A\rtimes_\alpha G$
are always Morita equivalent whenever $\dimRok(\alpha)<\I$
and $\alpha$ is unital. Our proof 
is very different from the arguments
used in \cite{GarHirSan_Rokhlin_2017},
and has the advantage of being applicable to global actions of groups
that are not necessarily abelian, thus obtaining new information even
in the global setting. 

In the next proof, for $\xi\in A\rtimes_\alpha G$ 
we write 
$\xi(g)\in A_g$, for 
$g\in G$, for the coefficient of $u_g$ in $\xi$, so that 
$\xi=\sum_{g\in G}\xi(g)u_g$.

\begin{thm}\label{thm:MoritaEqFPtCP}
Let $\alpha$ be a unital partial action of a finite 
group $G$ on a $C^*$-algebra $A$ with $\dimRok(\alpha)<\infty$. 
Then $A^\alpha$ is Morita equivalent to $A\rtimes_\alpha G$.
\end{thm}
\begin{proof} 
Set $x_\alpha=\sum_{g\in G}1_g$. Then $x_\alpha$ is central, 
positive, and since $x_\alpha\geq 1$, it is invertible. 
Moreover, $x_\alpha$ belongs to $A^\alpha$, since for $h\in G$
we have
\[\alpha_h(x_\alpha 1_{h^{-1}})=\sum_{g\in G}\alpha_h(1_g1_{h^{-1}})=
\sum_{g\in G} 1_{hg}1_h=x_\alpha 1_h,
\]
where at the second step we used \autoref{rem:UnitsIdeals}. 
We define a Hilbert 
$A^\alpha$-$A\rtimes_{\alpha} G$-bimodule structure on
$A$ as follows. For $a\in A^\alpha$, $x\in A$ and $\xi\in A\rtimes_{\alpha}G$, set
 $a\cdot x=ax$ (the product is taken in $A$) and 
$x\cdot \xi= \sum\limits_{g\in G} \alpha_{g^{-1}}(x\xi(g))$.
The inner products are given by
\[_{A^\alpha}\langle{x,y}\rangle = \sum\limits_{g\in G}\alpha_g(xy^*1_{g^{-1}}) \ \  \mbox{ and } \ \
\langle{x,y}\rangle_{A\rtimes_{\alpha}G}=\sum\limits_{g\in G} x^*\alpha_g(y1_{g^{-1}})u_g.\]
for $a\in A^\alpha$, $x,y\in A$ and $\xi\in A\rtimes_\alpha G$.
(One readily checks that $_{A^\alpha}\langle{x,y}\rangle$
belongs to $A^\alpha$.) The properties of the inner products
are easily verified. For example, 
for $x\in A$ we have
\[x_\alpha\sum\limits_{g\in G}x^*\alpha_g(x1_{g^{-1}})u_g=\bigg(\sum\limits_{g\in G}\alpha_{g}(x1_{g^{-1}})u_g\bigg)^*\bigg(\sum\limits_{g\in G}\alpha_{g}(x1_{g^{-1}})u_g\bigg)\geq 0,
\]
so $\langle{x,x}\rangle_{A\rtimes_{\alpha}G}=\sum\limits_{g\in G}x^*\alpha_g(x1_{g^{-1}})u_g\geq 0$, with equality exactly when $x=0$.
One also easily checks the properties of a bimodule.
For example, given $\xi\in A\rtimes_\alpha G$
and $x,y\in A$, we have 
\begin{align*}\langle{x,y}\rangle_{A\rtimes_{\alpha}G}\cdot \xi&=\sum\limits_{g\in G} x^*\alpha_g(y1_{g^{-1}})u_g\sum\limits_{h\in G}\xi(h)u_h\\
&=\sum\limits_{g,h\in G}\alpha_g(\alpha_{g^{-1}}(x^*\alpha_g(y1_{g^{-1}}))\xi(h))u_{gh}\\
&= \sum\limits_{g,h\in G}\alpha_g(\alpha_{g^{-1}}(x^*1_g)y\xi(h))u_{gh}\\
&=\sum\limits_{g,h\in G}x^*\alpha_g(y\xi(h)1_{g^{-1}})u_{gh}\\
&=\sum\limits_{g,h\in G}x^*\alpha_{gh^{-1}}(y\xi(h)1_{hg^{-1}})u_{g}\\
&= \sum\limits_{g,h \in G} x^*\alpha_g(\alpha_{h^{-1}}(y\xi(h))1_{g^{-1}})u_g\\
&= \langle{x,y\cdot \xi}\rangle_{A\rtimes_{\alpha}G}.\end{align*}
The other properties are shown similarly.
Moreover, for $x\in A^\alpha$, we have
\[_{A^\alpha}\langle{x,x_\alpha^{-1}}\rangle=\sum\limits_{h\in G}\alpha_h(xx_\alpha^{-1}1_{h^{-1}})=\sum\limits_{h\in G}xx_\alpha^{-1}1_h=x.\]
In particular, $A$ is full left $A^\alpha$-module.
Finally, we claim that 
$A$ is a full right $A\rtimes_{\alpha}G$-module. For this,
it suffices to show that given $t\in G$, $x\in A_t$
and $\ep>0$, there exist $k\in\N$ and $a_1,b_1,\ldots, a_k, b_k\in A$
such that $xu_t\approx_\ep \sum\limits_{\ell=1}^{k}\langle{a_\ell,b_\ell}\rangle_{A\rtimes_{\alpha}G}$.
Set $d=\dimRok(\alpha)$ and $\ep_0=\ep/2(d+1)|G|$, and use \autoref{prop:RdimReductionUnital}
to find $f_g^{(j)}\in A_g$, for $g\in G$
and $j=0,\ldots,d$, satisfying
\be
\item[(a)] $\alpha_g(f_h^{(j)}1_{g^{-1}})=f_{gh}^{(j)}1_g$, for all $g,h\in G$;
\item[(b)] $f_g^{(j)}f_h^{(j)}\approx_{\ep_0} 0$ for all $j=0,\ldots,d$ and all distinct $g,h\in G$;
\item[(c)] $\sum_{g\in G}\sum_{j=0}^d f_g^{(j)}\approx_{\ep_0}1$.
\ee
Using at the last step that 
$x\in A_t$, we get
\begin{align*}
\sum_{g\in G}\sum_{j=0}^d\langle{{f_{tg}^{(j)}}^{\frac{1}{2}}x^*,{f_g^{(j)}}^{\frac{1}{2}}}\rangle_{A\rtimes_{\alpha}G}
&=\sum_{g,h\in G}\sum_{j=0}^d x{f_{tg}^{(j)}}^{\frac{1}{2}}\alpha_h({f_g^{(j)}}^{\frac{1}{2}}1_{h^{-1}})u_h\\
&\stackrel{\mathrm{(a)}}{=}\sum_{g,h\in G}\sum_{j=0}^d x{f_{tg}^{(j)}}^{\frac{1}{2}}{f_{hg}^{(j)}}^{\frac{1}{2}}1_hu_h\\
&\stackrel{\mathrm{(b)}}{\approx}_{\! \frac{\ep}{2}}\sum_{g\in G}\sum_{j=0}^d xf_{tg}^{(j)}1_tu_t\\
&\stackrel{\mathrm{(c)}}{\approx}_{\ep_0} x1_tu_t=xu_t,
\end{align*}
as desired. We conclude that $A$ is an $A^\alpha$-$A\rtimes_{\alpha}G$-imprimitivity bimodule, so these $C^*$-algebras are Morita equivalent,
and the proof is complete.
\end{proof}


\begin{thm}\label{thm:CPRdimUnital}
Let $G$ be a finite group, and let $d\in \N$. 
Let \textbf{P} be a property for \ca s which is preserved by:
\bi\item[(S)] passage to direct sums and summands; 
\item[(M)] Morita equivalence;
\item[(C)] crossed products by \emph{global} actions
of $G$ with $\dimRok\leq d$. 
\ei
Let $A$ be a unital \ca, and let 
$\alpha$ be a unital partial action of $G$ on $A$. If 
$\dimRok(\alpha)\leq d$ and $A$ satisfies
\textbf{P}, then so do $A\rtimes_\alpha G$ and $A^\alpha$. A similar statement
is true if $\dimRok$ is replaced everywhere by $\cdimRok$. 
In particular,
\be

\item[(1)] If $\cdimRok(\alpha)<\I$ and $A$ satisfies the UCT, then
so do $A\rtimes_\alpha G$ and $A^\alpha$.
\item[(2)] If $\cdimRok(\alpha)<\I$ and $K_\ast(A)$ is either trivial, free, torsion-free, or finitely-generated, then the same holds for $A\rtimes_\alpha G$ and $A^\alpha$. 
\ee

When $\dimRok(\alpha)=0$, 
it follows that the properties listed in
the main theorem of \cite{Gar_crossed_2017} pass from $A$ 
to $A^\alpha$ and $A\rtimes_\alpha G$. This includes
having real rank zero, having stable rank one, 
being an AF/AI/AT-algebra, being purely infinite, 
the order on projections being determined by traces, being
weakly semiprojective, and having $K$-groups which are either
trivial, free, torsion-free, or finitely-generated.
\end{thm}
\begin{proof}
By \autoref{thm:MoritaEqFPtCP} and condition (M), 
it suffices to show the result for $A\rtimes_\alpha G$.
Let $(B,\beta)$ denote the globalization of $(A,\alpha)$. 

\textbf{Claim:} \emph{there 
exist central projections $p_1,\ldots, p_{|G|}\in A$ such that
$B\cong p_1A\oplus\cdots\oplus p_{|G|}A$.}
Since $A$ is a unital ideal in $B$ and $B=\sum_{g\in G}\beta_g(A)$,
it suffices to show the following: if $I,J$ are unital ideals in a 
\ca, then $I+J\cong pI\oplus J$ for some central projection
$p\in I$. This follows by taking $p=1_I-1_I1_J$ and letting
$\psi\colon I+J\to pI\oplus J$ be $\psi(x)=(px,1_Jx)$
for all $x\in I+J$. We omit the easy details.

Assume that $A$ satisfies \textbf{P} and that $\dimRok(\alpha)\leq d$.
By the claim above and (S), it follows that $B$ satisfies \textbf{P}.
By \autoref{thm:RdimGlobaliz}, we have $\dimRok(\beta)=\dimRok(\alpha)\leq d$, and hence by (C) the global crossed product 
$B\rtimes_\beta G$ satisfies \textbf{P}. Finally, since
$B\rtimes_\beta G\sim_M A\rtimes_\alpha G$ by
Theorem~4.18 in~\cite{AbaGarGef_decomposable_2020}, the result follows from (M).

For the properties listed in (1) and (2), the preservation conditions
(S), (M) are well-known, while (C) is part of
Theorem~3.20 in~\cite{GarHirSan_Rokhlin_2017}. The properties 
mentioned in the last part of the statement are also known to 
satisfy (S) and (M), and for $d=0$ they also satisfy (C) by the 
main result in~\cite{Gar_crossed_2017}.
\end{proof}

\begin{rem}
In the context of the above theorem, one can obtain bounds
for $\dimnuc$, $\dr$, $\mathrm{sr}$ and $\mathrm{RR}$ of 
$A\rtimes_\alpha G$ and $A^\alpha$ 
that are much better than the ones in
\autoref{thm:CPRdim}, particularly since they do not depend
on the cardinality of $G$. For example, one gets
$\dimnuc(A\rtimes_\alpha G)\leq (\dimRok(\alpha)+1)(\dimnuc(A)+1)-1$.
\end{rem}

\section{Topological partial actions}\label{sec: ParDynSys}

In this section, we study the Rokhlin dimension of 
topological partial actions. In \autoref{thm:FreeFiniteRdim}, 
we will show that a topological partial action 
$G\curvearrowright X$ on a finite-dimensional space $X$
has finite Rokhlin dimension \emph{if and only if} it is free. 
The case of global actions is implicit 
in~\cite{HirPhi_rokhlin_2015}, 
and it is an easy consequence of the 
existence of local cross-sections for the quotient map 
$\pi\colon X\to X/G$.
However, the proof in the partial setting is considerably
more complicated, since even for free partial actions there
may not exist local cross-sections for $\pi$. The proof 
in our context is
quite involved and will occupy the entire section. 

\begin{df}
Let $X$ be a topological space. A topological partial action of a discrete group $G$ on $X$ is given by a pair $((X_g)_{g\in G}, (\theta_g)_{g\in G})$, consisting of open subsets $X_g\subseteq X$ and homeomorphisms $\theta_g\colon X_{g^{-1}}\to X_g$, satisfying
\be
\item $X_1=X$ and $\theta_1=\id_X$, and
\item $\theta_g\circ\theta_h\subseteq \theta_{gh}$, for all $g,h\in G$.
\ee 

We say that $\theta$ is \emph{free} if $\theta_{g}(x)\neq x$ for all $g\in G\setminus\{1_G\}$ and all $x\in X_{g^{-1}}$.
\end{df}

Next, we show that finite Rokhlin dimension implies freeness in the
above setting.

\begin{prop}\label{prop:RdimFreeness}
Let $X$ be a locally compact Hausdorff space,
let $G$ be a finite group, let $\theta$
be a topological partial action of $G$ on $X$, 
and denote by $\alpha$ the induced
partial action of $G$ on $C_0(X)$. 
If $\dimRok(\alpha)<\infty$, then $\theta$ is free.
\end{prop}

\begin{proof}
Set $d=\dimRok(\alpha)$, and assume that $d<\I$. Let
$g\in G\setminus \{1\}$ be given. Arguing by contradiction, assume that $F_g=\{x\in X_{g^{-1}}\colon \theta_{g}(x)=x\}$
is not empty, and fix $x\in F_g$. Note that $x$ belongs to $X_g\cap X_{g^{-1}}$. Let $a\in C_0(X_{g^{-1}})$ be a positive contraction
satisfying $a(x)=1$. Choose $\ep>0$ with $\ep<\frac{1}{2(|G|(d+1)+1)^2}$, and let $f_h^{(j)}\in C_0(X_h)$, for $h\in G$,
and $j=0,\ldots, d$ be Rokhlin towers for
$(\{a, \alpha_g(a)\},\ep)$.
Given $h\in G$ and $j=0,\ldots,d$, we have:
\begin{align*}
\left\|f_h^{(j)}\alpha_g(f_h^{(j)}a)\right\|_\I&\leq \|f_h^{(j)}\big(\underbrace{\alpha_g(f_h^{(j)}a)-f_{gh}^{(j)}\alpha_g(a)}_{\ \ \ \approx_\ep 0}\big)\|_\I+ \|\underbrace{f_h^{(j)}f_{gh}^{(j)}}_{ \ \ \ \approx_\ep 0}\alpha_g(a)\|_\I<2\ep.
\end{align*}
Evaluating at $x$, we get
\[(f_h^{(j)}\alpha_g(f_h^{(j)}a))(x) = f_h^{(j)}(x)f_h^{(j)}(\theta_{g^{-1}}(x))a(\theta_{g^{-1}}(x))<2\ep.\]
Since $\theta_{g^{-1}}(x)=x$ and $a(x)=1$, we deduce that
\[\label{eqn:5.1}\tag{5.1}f_h^{(j)}(x)<\sqrt{2\ep},\] 
for all $h\in G$ and $j=0,\ldots,d$.
Moreover, condition~(3) from~\autoref{df:Rdim} gives
\[
\left|\sum\limits_{j=0}^{d}\sum\limits_{h\in G}f_h^{(j)}(x)a(x)-a(x)\right|\stackrel{a(x)=1}{=}\left|\sum\limits_{j=0}^{d}\sum\limits_{h\in G}f_h^{(j)}(x)-1\right|<\ep,\]
so $\sum\limits_{j=0}^{d}\sum\limits_{h\in G}f_h^{(j)}(x)>1-\ep$.
Using this at the second step, we get
\[|G|(d+1)\sqrt{2\ep}\stackrel{(\ref{eqn:5.1})}{>}\sum\limits_{j=0}^{d}\sum\limits_{g\in G}f_g^{(j)}(x)
>1-\ep\geq 1-\sqrt{2\ep}.\]
This contradicts the choice of $\ep$, showing that $F_g$ is empty. Thus $\theta$ is free.
\end{proof}

The rest of the section
will be devoted to prove that free partial actions have finite
Rokhlin dimension. The general strategy is as follows:
\bi
\item[Step 1:] Show that free \emph{decomposable} partial actions have 
finite Rokhlin dimension.
\item[Step 2:] Show that an extension of topological partial actions with finite Rokhlin dimension again has finite Rokhlin dimension.
\item[Step 3:] Use the equivariant decomposition into 
successive extensions from \autoref{thm: FinGrpIntersProp}
to conclude that the given action has finite Rokhlin dimension.
\ei 

Using the results we proved in Section~4, we can
easily establish the first step.

\begin{prop}\label{prop:FreeRokDimDecomp}
Let $G$ be a finite group, let $n\in \{1,\ldots,|G|\}$, 
let $X$ be a locally compact space with $\dim(X)<\I$, 
and let 
$\sigma$ be a partial action of $G$ on $X$ with the
$n$-decomposition property (see \autoref{df:nIntProp}). Denote by
$\alpha$ the induced partial action of $G$ on $C_0(X)$.
If $\sigma$ is free, then $\dimRok(\alpha)\leq \dim(X)$.
\end{prop}
\begin{proof}
Fix 
$\tau\in \mathcal{T}_n(G)$. 
Then $X_\tau$ is an open subset of $X$ and thus $\dim(X_\tau)\leq 
\dim(X)$. By part~(2) of Theorem~5.4 in~\cite{AbaGarGef_decomposable_2020}, the restricted global action
of $H_\tau$ on $X_\tau$ is free. Using Proposition~2.11 in~\cite{HirPhi_rokhlin_2015} at the first step, it follows that 
\[\dimRok(\alpha|_{H_\tau})\leq \dim(X_\tau)\leq \dim(X).\]
Thus, the result follows from \autoref{thm:dimRokResH}.\end{proof}

Step 2 is considerably more complicated.
Roughly speaking, one needs to lift 
Rokhlin towers from the quotient to the algebra, 
while at the same time respecting the domains of the partial 
action. If we did not care about respecting the domains
(such as in the global case), 
the result would be an immediate consequence of the 
fact that the cone of $\C^n$ is projective. In our setting,
the greatest difficulty lies in showing that this can be 
done in a way compatible with the domains, and for this 
we will need some lemmas that allow us to assume that
the Rokhlin towers and their lifts are orthogonal; see 
\autoref{lma:CommAlgOrthogonalTowers} and \autoref{lma:liftingCn}. 
It is unclear whether
these results hold without assuming that the algebra is commutative.

First, we show that for partial actions on commutative \uca s, the elements of each Rokhlin tower can be 
assumed to be exactly orthogonal.  

\begin{lma}\label{lma:CommAlgOrthogonalTowers}
Let $\alpha$ be a partial action of a finite group $G$ on a unital, commutative $C^*$-algebra $A$, and let $d\in\N$.
Then $\dimRok(\alpha)\leq d$ if and only if for all $\ep>0$ and every finite subset $F\subseteq A$, there exist positive contractions 
$f_g^{(j)}\in A_g$, for $g\in G$ and $j=0,\ldots,d$, satisfying:
\be
\item $\left\| \alpha_g(f_h^{(j)}a)- f_{gh}^{(j)}\alpha_g(a)\right\|<\varepsilon$, for all $a\in F\cap A_{g^{-1}}$;
\item $f_g^{(j)}f_h^{(j)}=0$ for all $g,h\in G$ with $g\neq h$ and for all $j=0,\ldots,d$;
\item $\left\|\sum\limits_{g\in G}\sum\limits_{j=0}^{d}f_g^{(j)}-1\right\|<\ep$.
\ee 
\end{lma}

\begin{proof} 
It is clear that any action satisfying conditions (1), (2) and (3) in the statement has Rokhlin dimension at most $d$ (the approximate 
commutation condition from \autoref{df:Rdim} is vacuous since $A$ is commutative). 
We therefore prove the non-trivial implication. 
Let $\ep>0$ and let a finite subset $F\subseteq A$ be given. 
Without loss of generality, we assume that $F$ consists of positive contractions. 
	
Fix $\ep_0<1$ such that $(d+2)|G|^2\sqrt{\ep_0}<\ep$, and 
choose positive contractions $x_g^{(j)}\in A_g$, for $g\in G$ and $j=0,\ldots,d$, satisfying
\be
\item[(a)] $\left\| \alpha_g(x_h^{(j)}a)- x_{gh}^{(j)}\alpha_g(a)\right\|<\ep_0$, for all $a\in F\cap A_{g^{-1}}$;
\item[(b)] $\left\|x_g^{(j)}x_h^{(j)}\right\|<\ep_0$ for all $g,h\in G$ with $g\neq h$ and for all $j=0,\ldots,d$;
\item[(c)] $\left\|\sum\limits_{g\in G}\sum\limits_{j=0}^{d}x_g^{(j)}-1\right\|<\ep_0$.
\ee 

For $g\in G$ and $j=0,\ldots,d$, set $y_g^{(j)}=\sum\limits_{h\in G\setminus\{g\}}x_h^{(j)}$ and 
\[f_g^{(j)}=(x_g^{(j)}-y_g^{(j)})_+.\]
Since $A$ is commutative, we have $f_g^{(j)}\leq x_g^{(j)}$ and hence $f_g^{(j)}$ is a positive contraction in $A_g$. We will show that these elements satisfy the conditions in the statement. We begin by noting that the $f_g^{(j)}$ are 
very close to the $x_g^{(j)}$.

\textbf{Claim:} \emph{we have $\left\|x_g^{(j)}-f_g^{(j)}\right\|< |G|\sqrt{\ep_0}$ for all $g\in G$ and $j=0,\ldots,d$.}
To see this, and since $A$ is abelian, it is enough to prove that 
$\left|x_g^{(j)}(t)-f_g^{(j)}(t)\right|\leq |G|\sqrt{\ep_0}$ for every $t\in \mathrm{Spec}(A)$. Fix $t\in \mathrm{Spec}(A)$.
We divide the proof into two cases.

Suppose that $x_g^{(j)}(t)\geq y^{(j)}_g(t)$.
Then $f_g^{(j)}(t)=x^{(j)}_g(t)-y_g^{(j)}(t)$ and hence
\[\label{eqn:5.2}\tag{5.2}
x_g^{(j)}(t)-f_g^{(j)}(t)=y_g^{(j)}(t)=\sum\limits_{h\in G\setminus\{g\}}x_h^{(j)}(t).\]
For $h\neq g$, we have 
$x_h^{(j)}(t)\leq y_g^{(j)}(t)\leq x_g^{(j)}(t)$,
and thus 
\[\label{eqn:5.3}\tag{5.3}x_h^{(j)}(t)^2\leq x_h^{(j)}(t)x_g^{(j)}(t)< \ep_0.\] 
We conclude that
\[\left|x_g^{(j)}(t)-f_g^{(j)}(t)\right|
\stackrel{(\ref{eqn:5.2})}{\leq}
\sum_{h\in G\setminus\{g\}}\left|x_h^{(j)}(t)\right|
\stackrel{(\ref{eqn:5.3})}{<} (|G|-1)\sqrt{\ep_0}< |G|\sqrt{\ep_0}.\]

Suppose that $x_g^{(j)}(t)\leq y^{(j)}_g(t)$. 
Then $f_g^{(j)}(t)=0$ and hence $x_g^{(j)}(t)-f_g^{(j)}(t)=x_g^{(j)}(t)$.
Moreover, 
\[\label{eqn:5.4}\tag{5.4}
x_g^{(j)}(t)^2\leq \left(\sum\limits_{h\in G\setminus\{g\}}x_h^{(j)}(t)\right)x_g^{(j)}(t)\stackrel{\mathrm{(b)}}{\leq} (|G|-1)\ep_0.\] 
Thus,
\[\left|x_g^{(j)}(t)-f_g^{(j)}(t)\right|=\left|x_g^{(j)}(t)\right|
\stackrel{(\ref{eqn:5.4})}{\leq}\sqrt{(\left| G\right| -1)\ep_0}< |G|\sqrt{\ep_0}.\]
This proves the claim.

We check that the positive contractions $f_g^{(j)}$ for $g\in G$ and $j=0,\ldots,d$ satisfy conditions (1), (2) and (3) in the statement. Conditions (1) and (3) follow immediately by 
combining the claim above with conditions (a) and (c), respectively, so we only check (2). 
Let $g,h\in G$ with $g\neq h$ and let $j=0,\ldots,d$. 
Using the inequalities
$x_g^{(j)}\leq y_h^{(j)}$ and $x_h^{(j)}\leq y_g^{(j)}$ (which are valid since $g\neq h$) at the second step; and the identity $(z-w)_+(w-z)_+=0$ at the last step, we get  
\begin{align*}
f_g^{(j)}f_h^{(j)}&= (x_g^{(j)}-y_g^{(j)})_+(x_h^{(j)}-y_h^{(j)})_+\\
&\leq (y_h^{(j)}-y_g^{(j)})_+(y_g^{(j)}-y_h^{(j)})_+=0
\qedhere
\end{align*}
\end{proof}

The following lemma deals with obtaining domain-respecting orthogonal lifts.

\begin{lma}\label{lma:liftingCn}
Let $A$ be a commutative \ca, let $n\in\N$, and let $J, A_1,\ldots, A_n$ be ideals in $A$. We set $B=A/J$ with quotient map $\pi\colon A\to B$, and 
\[B_j=\frac{A_j}{A_j\cap J}\cong \frac{A_j+J}{J}\]
for $j=1,\ldots,n$. For $j=1,\ldots,n$, let $x_j\in A_j$ be a positive contraction satisfying $x_jx_k\in J$ for all $k=1,\ldots,n$ with $j\neq k$. 
Then there exist pairwise orthogonal positive contractions $y_j\in A_j$, for $j=1,\ldots,n$, satisfying $\pi(y_j)=\pi(x_j)$.
\end{lma}

\begin{proof} 
We prove the lemma by induction on $n$, the case $n=1$ being trivial. Assume that $n=2$, and let $x_1\in A_1$ and $x_2\in A_2$ be positive contractions
satisfying $x_1x_2\in J$. Set $y_1=(x_1-x_2)_+$ and $y_2=(x_2-x_1)_+$, which are positive contractions. 
Since $A$ is abelian, we have $y_j\leq x_j$ and hence $y_j$ belongs to $A_j$, for $j=1,2$. Moreover, $y_1y_2=0$.
Finally, using at the first step that $\pi(x_1)\pi(x_2)=0$, we get
\[\pi(x_{1})=(\pi(x_{1})-\pi(x_{2}))_+=\pi((x_1-x_2)_+)= \pi(y_{1}).\]
Similarly, $\pi(x_2)=\pi(y_2)$.
This proves the case $n=2$ of the statement. 

Assume now that we have proved the result for an integer $n\geq 2$, and let us prove it for $n+1$. Thus, let $A_1,\ldots, A_{n+1}$ be ideals in $A$, and let 
$x_1,\ldots, x_{n+1}$ be as in the statement. Set $x=\sum_{j=1}^n x_j$ and define $y_{n+1}=(x_{n+1}-x)_+$. Then $y_{n+1}$ is a positive contraction
with $y_{n+1}\leq x_{n+1}$, and hence $y_{n+1}\in A_{n+1}$. Moreover, since $\pi(x)\pi(x_{n+1})=0$ and using the same argument as in the case $n=2$, 
it follows that $\pi(y_{n+1})=\pi(x_{n+1})$. 

For $j=1,\ldots,n$, set $z_j=(x_j-x_{n+1})_+$, which is a positive contraction in $A_j$ with $z_j\leq x_j$. Since $\pi(x_j)\pi(x_{n+1})=0$ for $j\leq n$,
it follows that $\pi(z_j)=\pi(x_j)$. In particular, for $1\leq j\neq k\leq n$ we have $\pi(z_jz_k)=0$, that is, $z_jz_k \in J$. By the inductive step
applied to $z_1,\ldots,z_n$, there exist orthogonal positive contractions $y_j\leq z_j$, for $j=1,\ldots,n$, satisfying $\pi(y_j)=\pi(z_j)$. 
Then $y_j\leq x_j$ and $\pi(y_j)=\pi(x_j)$ for $j=1,\ldots,n+1$. It remains to prove that the $y_1,\ldots,y_{n+1}$ are pairwise orthogonal. 
By construction, the contractions $y_1,\ldots, y_n$ are pairwise orthogonal, so it suffices to check orthogonality with $y_{n+1}$. For $j=1,\ldots,n$, 
we have
\[y_jy_{n+1}\leq z_j y_{n+1} = (x_j-x_{n+1})_+ (x_{n+1}-x)_+\leq (x_j-x_{n+1})_+ (x_{n+1}-x_j)_+=0,\]
using at the third step that $A$ is commutative and $x\geq x_j$. 
\end{proof}

We are now ready to prove that an extension of partial actions on commutative \ca s with finite Rokhlin dimension
again has finite Rokhlin dimension.
For use in its proof, we recall the following standard fact. 
Let $A$ be a \uca, let $J$ be an ideal in $A$, let $\pi\colon A\to A/J$ denote the quotient map, let $(e_\lambda)_{\lambda\in\Lambda}$
be any approximate identity for $J$, and let $a\in A$. 
Then $\|\pi(a)\|=\limsup\limits_{\lambda\in\Lambda}\|a(1-e_\lambda)\|$.

\begin{prop}\label{prop:ExtRdim}
Let $G$ be a finite group, let $A$ be a commutative unital $C^*$-algebra,
let $\alpha$ be a partial action of $G$ on $A$, and let $J$ be an $\alpha$-invariant ideal in $A$ with 
a $G$-invariant approximate identity. 
Denote by $\overline{\alpha}$ the induced partial action of $G$ on $A/J$ as in \autoref{prop:PermProp}. Then 
\[\dimRok(\alpha)\leq \dimRok(\alpha|_{J})+\dimRok(\overline{\alpha})+1.\]
\end{prop}

\begin{proof}
Set $d_1=\dimRok(\alpha|_J)$ and $d_2=\dimRok(\overline{\alpha})$. 
Without loss of generality, we assume that $d_1,d_2<\I$. Let $\ep>0$ and let $F\subseteq A$ be a finite subset. 
We abbreviate $B=A/J$ and $B_g= A_g/(A_g\cap J)$ for all $g\in G$.
We denote by $\pi\colon A\to B$ the canonical quotient map. 
We use \autoref{lma:CommAlgOrthogonalTowers} 
to choose positive contractions $x_g^{(j)}\in B_g$, for $g\in G$ and $j=0,\ldots, d_2$,
satisfying 
\be
\item[(B.1)] $\left\| \overline{\alpha}_g(x_h^{(j)}b)- x_{gh}^{(j)}\overline{\alpha}_g(b)\right\|<\varepsilon$, for all $g,h\in G$, for all $j=0,\ldots,d_2$, 
and for all $b\in \pi(F)\cap B_{g^{-1}}$;
\item[(B.2)] $x_g^{(j)}x_h^{(j)}=0$, for all $g,h\in G$ with $g\neq h$ and for all $j=0,\ldots,d_2$;
\item[(B.3)]  $\left\|\sum\limits_{g\in G}\sum\limits_{j=0}^{d_2}x_g^{(j)}-1\right\|<\ep$. 
\ee
	 
Fix $j=0,\ldots,d_2$. Use \autoref{lma:liftingCn} to find mutually orthogonal positive contractions $y_g^{(j)}\in A_g$, for $g\in G$,
satisfying $\pi(y^{(j)}_g)=x_g^{(j)}$. Using that 
$J$ has a $G$-invariant approximate identity, find a positive contraction $q\in J$ 
satisfying:
\be
\item[(a)] $\alpha_g(qx)=q\alpha_g(x)$, for all $x\in A_{g^{-1}}\cap J$;
\item[(b)] $\left\| (1-\sum\limits_{g\in G}\sum\limits_{j=0}^{d} y_g^{(j)})(1-q)\right\|<\ep $;
\item[(c)] $\left\|  (\alpha_g(y_h^{(j)}a)-y_{gh}^{(j)}\alpha_g(a))(1-q)\right\|<\ep$, for all $a\in A_{g^{-1}}\cap F$.
\ee

\textbf{Claim:} \emph{$\alpha_g(qa)=q\alpha_g(a)$ for all $a\in A_{g^{-1}}$.}
To prove this, let $(e_\lambda)_{\lambda\in \Lambda}$ be an approximate identity for $A_{g^{-1}}\cap J$. 
Then $(\alpha_g(e_\lambda))_{\lambda\in \Lambda}$ is an approximate identity for $A_g\cap J$. 
Fix $a\in A_{g^{-1}}$. In the next computation, we use at the third step condition (a) above with $x=ae_\lambda$; 
and the fact that $qa$ belongs to $A_{g^{-1}}\cap J$ at the fourth step: 
\[q\alpha_g(a)=\lim_{\lambda\in \Lambda}q\alpha_g(a)\alpha_g(e_\lambda)=\lim_{\lambda\in\Lambda}q\alpha_g(ae_\lambda)= \lim_{\lambda\in\Lambda}\alpha_g(qae_{\lambda})=\alpha_g(qa).\]
This proves the claim.

Set $z_g^{(j)}=y_g^{(j)}(1-q)$ for $g\in G$ and $j=0,\ldots,d_2$. 
Then 
\[\label{eqn:5.5}\tag{5.5}z_g^{(j)}z_h^{(j)}=0\] for $g,h\in G$ with $g\neq h$ and for $j=0,\ldots,d_2$, since $1-q$ is central. 
Let $g,h\in G$, let $a\in F\cap A_{g^{-1}}$, and let $j=0,\ldots,d_2$ be given. 
Using the above claim at the second step, we get  
\begin{align*}\label{eqn:5.6}\tag{5.6}
\left\|\alpha_g(z_h^{(j)}a)-z_{gh}^{(j)}\alpha_g(a)\right\|&=\left\|\alpha_g(y_h^{(j)}a(1-q))-y_{gh}^{(j)}\alpha_g(a)(1-q)\right\|\\
&=\left\|\left(\alpha_g(y_h^{(j)}a)-y_{gh}^{(j)}\alpha_g(a)\right)(1-q)\right\|
\stackrel{(c)}{<}\ep. 
\end{align*}
Furthermore, 
\[\label{eqn:5.7}\tag{5.7}\left\|\sum\limits_{g\in G}\sum\limits_{j=0}^{d_2}z_g^{(j)}-(1-q)\right\|<\ep.\]

Set $F_J=\{qa\colon a\in F\}\cup \{q\}\subseteq J$.
Choose positive contractions $\xi_g^{(k)}\in A_g\cap J$, 
for $g\in G$ and $k=0,\ldots, d_1$, satisfying
\be
\item[(J.1)] $\left\| \alpha_g(\xi_h^{(k)}c)- \xi_{gh}^{(k)}\alpha_g(c)\right\|<\varepsilon$, for all $g,h\in G$, for all $k=0,\ldots,d_1$, 
and for all $c\in F_J\cap J_{g^{-1}}$;
\item[(J.2)] $\left\|\xi_g^{(k)}\xi_h^{(k)}c\right\|<\ep$, for all $g,h\in G$ with $g\neq h$, for all $k=0,\ldots,d_1$, and for all $c\in F_J$;
\item[(J.3)]  $\left\|c-c\left(\sum\limits_{g\in G}\sum\limits_{k=0}^{d_1}\xi_g^{(k)}\right)\right\|<\ep$ for all $c\in F_J$. 
\ee

Set $\eta_g^{(k)}=\xi_g^{(k)}q$ for all $g\in G$ and $k=0,\ldots, d_1$.
For $a\in F\cap A_{g^{-1}}$, we have 
\[\label{eqn:5.8}\tag{5.8}\left\| \alpha_g(\eta_h^{(k)}a)-\eta_{gh}^{(k)}\alpha_g(a)\right\|=\left\| \alpha_g(\xi_h^{(k)}qa)-\xi_{gh}^{(k)}q\alpha_g(a)\right\|\stackrel{\textrm{(J.1)}}{<} \ep,\]
because $qa\in J_{g^{-1}}\cap F_J$.
Since $q$ is central (because $A$ is commutative), we have 
\[\label{eqn:5.9}\tag{5.9}\left\|\eta_g^{(k)}\eta_h^{(k)}\right\|\leq \left\| \xi_g^{(k)}\xi_h^{(k)}\right\|<\ep
\]
for $g,h\in G$ with $g\neq h$ and for $k=0,\ldots,d_1$.
Further, taking $c=q$ in (J.3) gives 
\[\label{eqn:5.10}\tag{5.10}\left\| \sum\limits_{g\in G}\sum\limits_{k=0}^{d_1} \eta_g^{(k)}-q\right\| \leq \ep.\]

For $g\in G$ and $\ell=0,\ldots,d_1+d_2+1$, set
\[f_g^{(j)} 
= \left.
\begin{cases}
z_g^{(j)}, & \text{for }   \ell=0,\ldots,d_2 \\
\eta_g^{(\ell-d_2-1)}, & \text{for } \ell=d_2+1,d_2+2,\ldots,d_1+d_2+1
\end{cases}
\right.\]
Then $\{f_g^{(j)}\colon g\in G, 0\leq \ell\leq d_1+d_2+1\}$ are Rokhlin towers for $\alpha$ with
respect to $(F,2\ep)$. Indeed, condition (1) follows from
(\ref{eqn:5.5}) and (\ref{eqn:5.8}); condition (2) 
follows from (\ref{eqn:5.6}) and (\ref{eqn:5.9}); and
condition (3) follows from (\ref{eqn:5.7}) and 
(\ref{eqn:5.10}).
\end{proof}

For technical reasons it will be more convenient to work with unital \ca s. 
In order to reduce to the unital case, we define the minimal partial unitization of a partial action. 
This unitization differs from the usual unitization of a global action, 
since the minimal partial unitization is never a global action, even if the original action was. 
For a \ca\ $A$, we denote by $A^+$ the unitization of $A$, which as a Banach space is isomorphic to $A\oplus \C$.

\begin{df} \label{df:UnitizationPA}
Let $G$ be a locally compact group, let $A$ be a \ca, and let $\alpha=\left((A_g)_{g\in G}, (\alpha_g)_{g\in G}\right)$ be a partial action of
$G$ on $A$. We define the \emph{minimal partial unitization} $\alpha^+$ of $\alpha$ to be the partial action of $G$ on $A^+$ determined by
$A^+_g=A_g$ and $\alpha^+_g=\alpha_g$ for all $g\in G\setminus\{1\}$.
\end{df} 

For later use, we record the following easy observation:

\begin{rem}\label{cor:UnitizRdim}
Let $G$ be a finite group, let $A$ be a commutative \ca, and let $\alpha$ be a partial action of $G$ on $A$.
Then $(A,\alpha)$ is an equivariant ideal in $(A^+,\alpha^+)$,
and hence 
$\dimRok(\alpha)\leq \dimRok(\alpha^+)$ by \autoref{prop:PermProp}.
\end{rem}



For a partial action $\theta$ of a group $G$ on a locally compact space $X$, one defines its \emph{minimal partial compactification} $\theta^+$ analogously to \autoref{df:UnitizationPA}, only by compactifying the domain of the identity element and leaving the rest unchanged. It is clear that the partial action of $G$ on $C(X^+)$ induced
by $\theta^+$ can be canonically identified by the minimal partial unitization of the partial action induced by $\theta$.

One feature of the minimal partial unitization, by comparison with the unitization in the sense of global actions, is that
the minimal partial unitization of a free action on a locally compact space is again free, as we show next. On the other
hand, the unitization (in the global sense) of a free action is never free, since the point at infinity is necessarily
fixed.

\begin{lma}\label{lma:UnitizPAfree}
Let $G$ be a locally compact group, let $X$ be a locally compact Hausdorff space, and let 
$\theta=\left((X_g)_{g\in G}, (\theta_g)_{g\in G}\right)$ be a partial action of $G$ on $X$.
Then $\theta$ is free if and only if $\theta^+$ is free. 
\end{lma}
\begin{proof}
It is clear that $\theta$ is free if $\theta^+$ is free. Conversely, and arguing by contradiction 
let $g\in G\setminus\{1\}$ and $x\in X_{g^{-1}}^+$ satisfy $\theta^+_g(x)=x$. 
Since $g\neq 1$, this implies that $x\in X_{g^{-1}}$ and $\theta_g(x)=x$, which contradicts the fact that 
$\theta$ is free, as desired. 
\end{proof}

We have now arrived at the main result of this section.

\begin{thm}\label{thm:FreeFiniteRdim}
Let $G$ be a finite group, let $X$ be a locally compact space
with $\dim(X)<\I$, let 
$\theta$ be a partial action of $G$ on $X$, and let $\alpha$
denote the partial action of $G$ on $C_0(X)$ induced
by $\theta$.
If $\theta$ is free, then 
\[\dimRok(\alpha)\leq (|G|-1)(\dim(X)+1).\]
\end{thm}
\begin{proof}
We begin with a general fact.

\textbf{Claim:} \emph{let $\beta$ be a free partial
action of $G$ on a unital commutative \ca\ $B$ with 
finite-dimensional spectrum. 
Suppose that there is an equivariant extension
\[\xymatrix{0\ar[r] & (D,\delta) \ar[r]& (B,\beta)\ar[r] & (C,\gamma)\ar[r]&0,}\]
with $\delta$ decomposable and $\dimRok(\gamma)<\I$. Then}
\[\dimRok(\beta)\leq \dimRok(\delta)+\dimRok(\gamma)+1<\I.\]
By Proposition~4.4 in \cite{AbaGarGef_decomposable_2020},
$(D,\delta)$ admits a $G$-invariant approximate identity.
Since $B$ is unital and commutative,
by \autoref{prop:ExtRdim}, 
we have $\dimRok(\beta)\leq \dimRok(\delta)+\dimRok(\gamma)+1$, so
it suffices to show that
$\dimRok(\delta)$ is finite. 
Since $\beta$ is free, it is easy to see that so is
$\delta$. 
Since $\delta$ is decomposable, 
the claim follows from \autoref{prop:FreeRokDimDecomp}. 

We now prove the theorem. We will obtain the bound for 
$\dimRok(\alpha^+)$,
which is enough by \autoref{cor:UnitizRdim}.
Since $\alpha^+$ is free by \autoref{lma:UnitizPAfree},
upon replacing $\alpha$ with $\alpha^+$ we may assume
that $X$ is compact. 
By Theorem~6.1 in~\cite{AbaGarGef_decomposable_2020},
there are canonical equivariant extensions 
\[\label{eqn:5.11}\tag{5.11}
\xymatrix{0\ar[r] & (D^{(k)},\delta^{(k)}) \ar[r]& (A^{(k)},\alpha^{(k)})\ar[r] & (A^{(k-1)},\alpha^{(k-1)})\ar[r]&0,}
\]
for $2\leq k \leq |G|$, satisfying the following properties 
be
\be 
\item $A^{(|G|)}=C(X)$ and $\alpha^{|G|}=\alpha$;
\item $\delta^{(k)}$ has the $k$-decomposition property;
\item $\alpha^{(1)}$ has the 1-decomposition property.
\ee

Note that all the partial actions involved are free
(this is shown inductively, using that restrictions
of free actions are free). Moreover, $A^{(k)}$ is unital
and commutative for all $k=2,\ldots,|G|$. 
Note that $\alpha^{(1)}$ has Rokhlin dimension zero
by \autoref{eg:trivialRp}.
Applying the above claim to the extension in
($\ref{eqn:5.11}$) with $k=2$, we get 
\[\dimRok(\alpha^{(2)})\leq \dimRok(\delta^{(2)})+1<\I.\] Continuing inductively, 
after $|G|-1$ steps we deducee that 
\[\label{eqn:5.12}\tag{5.12}
\dimRok(\alpha)\leq  \left(\dimRok(\delta^{(2)})+1\right)+\cdots+ \left(\dimRok(\delta^{(|G|)})+1\right)<\I.
\]

Suppose now that $\dim(X)$ is finite. For $k=2,\ldots,|G|$, 
let $Y_k$ denote the spectrum of $D^{(k)}$. 
Since $\dimRok(\delta^{(k)})\leq \dim(Y_k)$ by \autoref{prop:FreeRokDimDecomp}, and $\dim(Y_k)\leq \dim(X)$, 
the dimensional 
estimate in the statement follows from
($\ref{eqn:5.12}$).
\end{proof}

For \emph{clopen} partial actions (that is, partial actions with clopen domains), no dimensionality assumptions
on $X$ are needed. 

\begin{prop}
Let $G$ be a finite group, let $X$ be a (not necessarily finite-
dimensional) compact
Hausdorff space, and let $\theta$ be a free clopen partial action
of $G$ on $X$, and let $\alpha$ denote the induced partial action
of $G$ on $C(X)$. Then 
\[\dimRok(\alpha)<\min\{\dim(X)+1,\I\}.\]
In particular, $\dimRok(\alpha)<\I$ even if 
$\dim(X)=\I$. 
\end{prop}
\begin{proof}
Denote by $(C(Y),\beta)$ the globalization of $(C(X),\alpha)$;
see Proposition~3.5 in~\cite{Aba_enveloping_2003}, noting that
the graph of $\theta$ is closed. Note that 
$\dimRok(\alpha)=\dimRok(\beta)$ by \autoref{thm:RdimGlobaliz}.
Let $\sigma$ denote the action of $G$ on $Y$ induced by $\beta$.

We claim that $\sigma$ is free. Note that $Y=\bigcup_{g\in G}\sigma_g(X)$. Let $g\in G$ and $y\in Y$ satisfy $\sigma_g(y)=y$.
Choose $h\in G$ and $x\in X$ such that $y=\sigma_h(x)$, so
that $\sigma_{h^{-1}gh}(x)=x$. Thus $x$ belongs to 
$X\cap \sigma_{h^{-1}gh}(X)=X_{h^{-1}gh}$, which is where
$\sigma_{hgh^{-1}}$ agrees with $\theta_{hgh^{-1}}$.
Thus $\theta_{hgh^{-1}}(x)=x$, and since $\theta$ is free, this
implies that $hgh^{-1}=1_G$ and hence also $g=1_G$. Hence
$\sigma$ is free.

Since $Y$ is compact, it follows from Theorem~4.2 in~\cite{Gar_compact_2018}
that $\dimRok(\beta)<\I$, and thus $\dimRok(\alpha)<\I$.

We complete the proof by showing that $\dimRok(\alpha)\leq \dim(X)$. 
Note that $\dim(Y)=\dim(X)$ (this follows, for example, by the claim in the 
proof of \autoref{thm:CPRdimUnital}).
Using \cite[Lemma~1.9]{HirPhi_rokhlin_2015} at the second, we get
\[\dimRok(\alpha)=\dimRok(\beta)\leq \dim(Y)=\dim(X).\qedhere\]
\end{proof}


\begin{thebibliography}{10}

\bibitem{Aba_enveloping_2003}
{\sc F.~Abadie}, {\em Enveloping actions and {T}akai duality for partial
  actions}, J. Funct. Anal. 197 (2003), 14--67.

\bibitem{AbaGarGef_decomposable_2020}
{\sc F.~Abadie, E.~Gardella, and S.~Geffen}, {\em Decomposable partial
  actions}, J. Funct. Anal., to appear.

  
\bibitem{BroPed_algebras_1991}
{\sc L.~Brown, G.~Pedersen}, {\em {$C^*$}-algebras of real rank zero},
  J. Funct. Anal. 99 (1991), 131--149.



\bibitem{EllTom_regularity_2008}
{\sc G.~Elliott, and A.~Toms}, {\em Regularity properties in the classification program for separable amenable {$C^*$}-algebras}. Bull. Amer. Math. Soc. 45 (2008), 229--245.

\bibitem{Exe_circle_1994}
{\sc R.~Exel}, {\em Circle actions on
  {$C^*$}-algebras, partial automorphisms, and a generalized
  {P}imsner-{V}oiculescu exact sequence}, J. Funct. Anal. 122 (1994),
  361--401.


\bibitem{Exe_book_2017}
{\sc R.~Exel}, {\em Partial dynamical
  systems, {F}ell bundles and applications}, vol.~224 of Mathematical Surveys
  and Monographs, American Mathematical Society, Providence, RI, (2017).

\bibitem{Fer_constructions_2018}
{\sc D.~Ferraro}, {\em Construction of globalizations for partial actions on
  rings, algebras, {$C^*$}-algebras and {H}ilbert bimodules}, Rocky
  Mountain J. Math. 48 (2018), 181--217.
  
\bibitem{Gar_rokhlin_2017}
{\sc E.~Gardella}, {\em Rokhlin dimension for compact group actions}. 
Indiana U. Math. J. 66 (2017), 659--703.

\bibitem{Gar_crossed_2017}
{\sc E.~Gardella}, {\em Crossed products by compact group actions with the Rokhlin property}, J. Noncommut. Geom. 11 (2017), 1593--1626. 
  
\bibitem{Gar_compact_2018}
{\sc E.~Gardella}, {\em {Compact group actions with the {R}okhlin property}}, 
 Trans. Amer. Math. Soc. 371 (2019), 2837--2874.
 
\bibitem{GarHir_strongly_2018}
{\sc E.~Gardella and I.~Hirshberg}, {\em Strongly outer actions of
amenable groups on $\mathcal{Z}$-stable {$C^*$}-algebras},(2018).
\newblock Preprint, arXiv:1811.00447.

\bibitem{GarHirSan_Rokhlin_2017}
{\sc E.~Gardella, I.~Hirshberg, and L.~Santiago}, {\em {Rokhlin dimension:
  duality, tracial properties, and crossed products}}, Ergodic Theory Dynam. Systems. 41 (2021), 408--460. 
 
\bibitem{GarLup_applications_2018}
{\sc E.~Gardella and M.~Lupini}, {\em Applications of model theory
to C*-dynamics}, J. Funct. Anal. 275 (2018), 1889--1942.

\bibitem{GarSan_equivariant_2016}
{\sc E.~Gardella and L.~Santiago}, {\em Equivariant *-homomorphisms, Rokhlin constraints and equivariant UHF-absorption}. J. Funct. Anal. 270 (2016), 2543--2590.

\bibitem{HirPhi_rokhlin_2015}
{\sc I.~Hirshberg and N.~C. Phillips}, {\em {Rokhlin dimension: obstructions
  and permanence properties}}, Doc. Math. 20 (2015), 199--236.


\bibitem{HirWinZac_rokhlin_2015}
{\sc I.~Hirshberg, W.~Winter, and J.~Zacharias}, {\em {Rokhlin dimension and
  {$C^*$}-dynamics}}, Comm. Math. Phys. 335 (2015), 637--670.


\bibitem{Izu_finiteI_2004}
{\sc M.~Izumi}, {\em {Finite group actions on {$C^*$}-algebras with the
  {R}ohlin property. {I}}}, Duke Math. J. 122 (2004), 233--280.

\bibitem{KirWin_covering_2004}
{\sc E.~Kirchberg and W.~Winter}, {\em {Covering dimension and
  quasidiagonality}}, Internat. J. Math. 15 (2004), 63--85.

\bibitem{OsaPhi_crossed_2012}
{\sc H.~Osaka and N.~C. Phillips}, {\em {Crossed products by finite group
  actions with the {R}okhlin property}}, Math. Z. 270 (2012), 19--42.


\bibitem{Rie_dimension_1983}
{\sc M.~Rieffel}, {\em Dimension and stable rank in the K-theory of {$C^*$}-algebras},
Proc. London Math. Soc. (3) 46 (1983), 301--333.
  

\bibitem{SzaWuZac_rokhlin_2014}
{\sc G.~Szabo, J.~Wu, and J.~Zacharias}, {\em {Rokhlin dimension for actions of
  residually finite groups}}, Ergodic Theory Dynam. Systems 39 (2019), 2248--2304.

\bibitem{TomWin_strongly_2007}
{\sc A.~S. Toms and W.~Winter}, {\em {Strongly self-absorbing
  {$C^*$}-algebras}}, Trans. Amer. Math. Soc. 359 (2007), 3999--4029.

\bibitem{WinZac_nuclear_2010}
{\sc W.~Winter and J.~Zacharias}, {\em {The nuclear dimension of
  {$C^*$}-algebras}}, Adv. Math. 224 (2010), 461--498.

\end{thebibliography}
\end{document}